\documentclass{amsart}

\usepackage[colorlinks,citecolor=niceblue,linkcolor=niceblue]{hyperref}
\usepackage{comment}
\usepackage{esint}
\usepackage{amssymb, amsrefs}
\usepackage{tikz}
\usetikzlibrary{calc}
\usetikzlibrary{patterns}
\usepackage{graphicx}
\usepackage{enumitem}
\usepackage{array}
\usepackage[export]{adjustbox}

\renewcommand{\MultipleCiteKeyWarning}[2]{}

\newtheorem{theorem}{Theorem}
\newtheorem{lemma}[theorem]{Lemma}
\newtheorem{corollary}[theorem]{Corollary}
\newtheorem{proposition}[theorem]{Proposition}
\theoremstyle{definition}

\newtheorem{remark}[theorem]{Remark}

\newtheorem{definition}[theorem]{Definition}

\newcommand{\eref}[1]{(\ref{e.#1})}
\newcommand{\tref}[1]{Theorem \ref{t.#1}}
\newcommand{\lref}[1]{Lemma \ref{l.#1}}
\newcommand{\pref}[1]{Proposition \ref{p.#1}}
\newcommand{\cref}[1]{Corollary \ref{c.#1}}
\newcommand{\fref}[1]{Figure \ref{f.#1}}
\newcommand{\sref}[1]{Section \ref{s.#1}}
\newcommand{\partref}[1]{\ref{part.#1}}
\newcommand{\dref}[1]{Definition \ref{d.#1}}
\newcommand{\rref}[1]{Remark \ref{r.#1}}

\newcommand{\deflink}[2]{\hyperref[#1]{#2}}

\newcommand{\OQE}{\textup{(\deflink{d.viscosity_solution}{O})}}
\newcommand{\EE}{\textup{(\deflink{d.energy_solution}{E})}}

\numberwithin{theorem}{section}
\numberwithin{equation}{section}

\newcommand{\R}{\mathbb{R}}

\newcommand{\grad}{\nabla}

\def\XXint#1#2#3{{\setbox0=\hbox{$#1{#2#3}{\int}$ }
\vcenter{\hbox{$#2#3$ }}\kern-.6\wd0}}

\newcommand{\ep}{\varepsilon}
\newcommand{\e}{\varepsilon} 

\newcommand{\dist}{\operatorname{dist}}
\newcommand{\Diss}{\operatorname{Diss}}

\newcommand{\one}{\mathbf{1}}

\definecolor{niceblue}{rgb}{0,0,0.7}
\definecolor{darkgreen}{rgb}{0,0.4,0}

\def\strikethrough#1{\setbox0\hbox{#1}\rlap{#1}\hbox to \wd0{\hss\strikebox\hss}}
\def\strikebox{\vrule height 0.6\ht0 depth -0.4\ht0 width 1.1\wd0}

\begin{document}

\title{An obstacle approach to rate independent droplet evolution}
\author[W. M. Feldman]{William M Feldman}
\address{Department of Mathematics, University of Utah, Salt Lake City, Utah, 84112, USA}
\email{feldman@math.utah.edu}
\author[I. C. Kim]{Inwon C Kim} 
\address{Department of Mathematics, University of California, Los Angeles, California, 90095, USA}
\email{ikim@math.ucla.edu}
\author[N. Po\v{z}\'{a}r]{Norbert Po\v{z}\'{a}r}
\address{Faculty of Mathematics and Physics, Institute of Science and Engineering, Kanazawa University, Kakuma, Kanazawa 920-1192, Japan}
\email{npozar@se.kanazawa-u.ac.jp}
\begin{abstract} We consider a toy model of rate independent droplet motion on a surface with contact angle hysteresis based on the one-phase Bernoulli free boundary problem. We introduce a notion of solutions based on an obstacle problem. These solutions jump ``as late and as little as possible", a physically natural property that energy solutions do not satisfy. When the initial data is star-shaped, we show that obstacle solutions are uniquely characterized by satisfying the local stability and dynamic slope conditions. This is proved via a novel comparison principle, which is one of the main new technical results of the paper. In this setting we can also show the (almost) optimal $C^{1,1/2-}$-spatial regularity of the contact line.  This regularity result explains the asymptotic profile of the contact line as it de-pins via tangential motion similar to de-lamination. Finally we apply our comparison principle to show the convergence of minimizing movements schemes to the same obstacle solution, again in the star-shaped setting. 
\end{abstract}

\maketitle

\setcounter{tocdepth}{1}
\tableofcontents

\section{Introduction}

The purpose of this paper is the analysis of geometric properties of the solutions of a toy model of a quasi-static droplet motion with contact angle hysteresis effects. 

Capillary surfaces incident to a solid surface are often subject to a phenomenon known as contact angle hysteresis. Instead of a single stable contact angle determined by the material properties, as predicted by the classical Young's law, there is a pinning interval, or range of stable apparent contact angles.  Thus, under small forcings, the contact line can ``stick" in a similar way to classical static friction in mechanics. The origin of contact angle hysteresis and the appropriate modelling of its effects remain the subject of much interest in the physics, engineering, and mathematics literature \cite{KarimReview}.

In the present model, we consider a droplet on a flat $d$-dimensional plane whose free surface is given as a graph of a function $u = u(t, x)$, $u: [0, \infty) \times U \to [0, \infty)$, where $U \subset \R^d$ is a connected domain with compact complement; see \fref{one-phase}. The sets $\Omega(u(t)) := \{u(t) > 0\}$ and $\{u(t) = 0\}$ are respectively the wet and dry regions at time $t$. At the domain boundary $\partial U$, the height of the droplet surface is a given function of time $F = F(t)$. We assume that the time scale at which the droplet reaches equilibrium is much shorter than the time scale at which $F$ changes, that is, the evolution is \emph{quasi-static}. Therefore at each time $u(t)$ minimizes the linearized surface area of the graph $\{(x, u(t, x)): u(t, x) > 0\}$, and at the boundary of the wet region $\partial \{u(t) > 0\} \cap U$ the contact slope is within the allowed interval $|\nabla u(t)|^2 \in[1 - \mu_-, 1 + \mu_+]$ with some $\mu_- \in (0,1)$ and $\mu_+ > 0$. In summary, at each time $t \in [0, T]$, $u(t)$ is a solution of the \emph{local stability condition}
\begin{equation}\label{e.stability-condition-intro}
\begin{cases}
\Delta u(t)= 0 & \hbox{ in } \ \Omega(u(t))\cap U,\\
1-\mu_- \leq |\grad u(t)|^2 \leq 1+\mu_+ & \hbox{on } \partial \Omega(u(t)) \cap U,
\end{cases}
\end{equation}
while it satisfies the Dirichlet boundary condition
\begin{align*}
u(t) = F(t) \qquad \text{on } \partial U.
\end{align*}

\begin{figure}
\centering
\begin{minipage}{0.48\textwidth}
\begin{tikzpicture}
\draw[->,dotted] (-1,0) -- (5.5,0) node[right] {$x$};
\draw[->,dotted] (0,0) -- (0,3.5) node[above] {$u$};
\draw[thick] plot [smooth] coordinates {
(0.0,3.0)
(0.4,2.45)
(0.8,2.0)
(1.2,1.62)
(1.6,1.29)
(2.0,1.0)
(2.4,0.75)
(2.8,0.529)
(3.2,0.333)
(3.6,0.158)
(4.0,0.0)
};
\fill[black!5!white] plot [smooth] coordinates {
(0.0,3.0)
(0.4,2.45)
(0.8,2.0)
(1.2,1.62)
(1.6,1.29)
(2.0,1.0)
(2.4,0.75)
(2.8,0.529)
(3.2,0.333)
(3.6,0.158)
(4.0,0.0)
} -- (0,0);
\fill[black!20!white] (0,0) -- (5,0) -- (5,-1) -- (0,-1);
\draw (1,1) node {liquid};
\draw (3,1.5) node {gas};
\draw (2.5,-0.7) node {solid};
\draw (2,0) node[below] {$\{u > 0\}$};
\draw (4,0) node[below] {$\partial \{u > 0\} \cap U$};
\fill (4,0) circle[radius=2pt];
\draw[thick] (0,0)--(4,0);
\draw (0,3) node[left] {$F$};
\draw (0.8,2.0) node[above right] {$u = u(x)$};
\draw (-0.5,0) node[below] {$U^\complement$};
\end{tikzpicture}
\end{minipage}
\hfill
\begin{minipage}{0.45\textwidth}
\begin{tikzpicture}
\draw[thick,fill=black!5!white] plot [smooth cycle] coordinates {
(1.96,0.0)
(1.95,0.633)
(1.58,1.15)
(1.2,1.65)
(0.607,1.87)
(0.0,2.02)
(-0.616,1.9)
(-1.17,1.61)
(-1.64,1.19)
(-1.87,0.607)
(-2.04,0.0)
(-1.85,-0.603)
(-1.66,-1.2)
(-1.15,-1.58)
(-0.629,-1.94)
(-0.0,-1.98)
(0.62,-1.91)
(1.18,-1.63)
(1.6,-1.16)
(1.94,-0.629)
};
\draw[fill=white] plot [smooth cycle] coordinates {
(1.2,0.0)
(0.759,0.551)
(0.259,0.797)
(-0.359,1.1)
(-0.859,0.624)
(-0.8,0.0)
(-0.859,-0.624)
(-0.359,-1.1)
(0.259,-0.797)
(0.759,-0.551)
}
;
\draw (1.5,-1.5) node[right] {$\partial \{u > 0\} \cap U$};
\draw (0,0) node {$U^\complement$};
\draw (0,-0.5) node {$u = F$};
\draw (1,1) node {$\{u > 0\}$};
\end{tikzpicture}
\end{minipage}
\caption{Side view (left) and the top view (right) of the setup for the one-phase free boundary problem.}
\label{f.one-phase}
\end{figure}
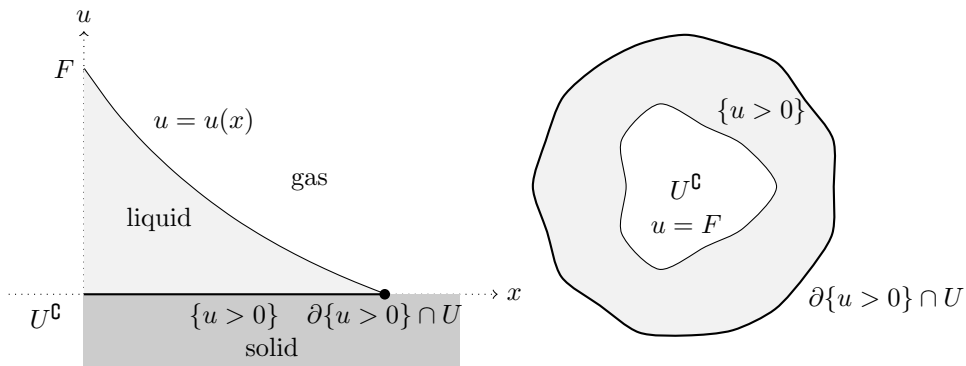

It remains to specify how the droplet reacts to the changes in the boundary condition $F(t)$. Here it seems reasonable to require that the support $\Omega(u(t)) = \{u(t) > 0\}$ adjusts as little as necessary so that the conditions in \eqref{e.stability-condition-intro} can be satisfied. 
This should mean that it expands (resp. shrinks) only at points where the contact angle condition $|\nabla u|^2 \leq 1 + \mu_+$ (resp. $|\nabla u|^2 \geq 1 - \mu_-$) saturates. This heuristic ``motion law'' suggests the \emph{dynamic slope condition}
\begin{equation}\label{e.dynamic-slope-condition-intro}
|\grad u(t)|^2(x) = 1\pm \mu_\pm \quad \hbox{ if } \quad \pm V_n(\Omega(u(t)),x) >0  \quad \hbox{on } \  \partial \Omega(u(t)) \cap U,
\end{equation}
where $V_n(\Omega(u(t)),x)$ is the outward normal velocity of $\Omega(u(t))$ at $x \in \partial \Omega(u(t))$.

It is natural to frame this as an obstacle problem for $u(t)$ as long as the forcing $F$ is piecewise monotone. More precisely, we assume that there is a finite set $Z$ such that \begin{equation}\label{e.forcing}
F: [0,T] \to (0,\infty) \hbox{ is Lipschitz and changes monotonicity only on } Z.
\end{equation}

The above motivates the following definition.

\begin{definition}\label{d.viscosity_solution}
We say that $u: [0, T] \times \overline U \to [0, \infty)$ is a \emph{obstacle solution} \OQE{} in $U$ driven by $F$ on $\partial U$ if
\begin{enumerate}
\item (\emph{Initial data}) $u(0)$ is a viscosity solution of the local stability conditions \eref{stability-condition-intro}.
\item (\emph{Dirichlet forcing})   For all $t \in [0,T]$
\begin{equation}\label{e.dirichletforcing-01}
u(t) = F(t) \ \hbox{ on } \ \partial U.
\end{equation}
\item (\emph{Obstacle condition}) For every $(s,t) \cap Z = \emptyset$, so that $F$ is monotone on $[s,t]$, $u(t)$ is the minimal supersolution of \eref{stability-condition-intro} and \eref{dirichletforcing-01} above $u(s)$ when $F$ is increasing on $[s,t]$ (resp. maximal subsolution below $u(s)$ when $F$ is decreasing).

\end{enumerate}
\end{definition}

The notions of minimal supersolution and maximal subsolution above are in the viscosity solutions / Perron's method sense; see \sref{comparison} for more details. The stability condition on the initial data is for convenience, otherwise the solution could jump at the initial time resulting in a ``replacement" initial data satisfying \eqref{e.stability-condition-intro}.

It is not hard to show that an obstacle solution satisfies the local stability and dynamic slope conditions. However, the question whether the conditions \eqref{e.stability-condition-intro} and \eqref{e.dynamic-slope-condition-intro} uniquely characterize the obstacle solution is less obvious and we give an affirmative answer in the strongly star-shaped setting with additional regularity assumptions. In fact, we do not expect this to be the case when there are jumps in the evolution, which cannot be easily ruled out in the non-star-shaped setting.

 \begin{theorem}[see \tref{equivalence}]\label{part.main1-p4}
 Let $F$ satisfy \eqref{e.forcing}, and let $u$ be an obstacle solution in the sense of \dref{viscosity_solution}:
 \begin{enumerate}[label = (\roman*)]
     \item If $u$ is uniformly non-degenerate at its free boundary (which holds if $d=2$ or if $\Omega_0$ is strongly star-shaped), then $u$ satisfies the local stability condition \eqref{e.stability-condition-intro} and the dynamic slope condition \eqref{e.dynamic-slope-condition-intro} in the viscosity solutions sense; \dref{local-stability-visc} and \dref{local-stability-visc}.
     \item  Furthermore, if $\Omega_0$ is a strongly star-shaped $C^{1,\alpha}$ domain, then any viscosity solution of \eqref{e.stability-condition-intro} and \eqref{e.dynamic-slope-condition-intro} is the unique obstacle solution.
 \end{enumerate}

 \end{theorem}

In a general setting, the obstacle solution \OQE{} can jump in time due to topological changes of the wet region $\Omega(u(t))$; see \sref{examples} for examples. The handling of jumps seems physically reasonable as, in a certain sense, the solution jumps ``as late as possible''. We believe that this is an important feature of this notion. This is in a contrast to an alternative approach to modeling the quasi-static evolution: a rate-independent evolution of the Bernoulli functional
\begin{equation}\label{e.energy-intro}
\mathcal{J}(v) = \int_U |\nabla v|^2 + {\bf 1}_{\{v>0\}} \ dx, \qquad v \in H^1(U),
\end{equation}
where $\one_{\{v > 0\}}$ is the indicator function of the set $\Omega(v) := \{v>0\}$,
with a \emph{dissipation distance}
\[\Diss(v_1,v_2) := \mu_+|\Omega(v_1) \setminus \Omega(v_0)| + \mu_-|\Omega(v_0) \setminus \Omega(v_1)|,\]  
again driven by the Dirichlet forcing $F$.

The authors explored this approach in a companion paper \cite{FKPi} so we refer the reader to that paper for details. Here we briefly review the important definitions. The approach is inspired by the work of DeSimone, Grunewald, and Otto \cite{DeSimoneGrunewaldOtto}, with theory developed by Alberti and DeSimone \cite{alberti2011}. Those works considered the surface energy capillarity functional with volume forcing using the energetic framework of rate-independent systems \cite{mielke2015book}.

In \cite{FKPi} the following notion of a solution was introduced:

\begin{definition}\label{d.energy_solution}
A measurable $u : [0,T] \to H^1(U)$ is a \emph{energy solution} \EE{} of the quasi-static evolution problem driven by Dirichlet forcing $F$ if the following hold:
\begin{enumerate}
\item (\emph{Forcing})  For all $t \in [0,T]$
\[ u(t) = F(t) \ \hbox{ on } \partial U\]
\item (\emph{Global stability}) The solution $u(t)\in H^1(U)$ and satisfies  for all $t \in [0,T]$:
\begin{equation}
\label{e.stability}
\mathcal{J}(u(t)) \leq \mathcal{J}(u') + \Diss(u(t),u') \qquad \text{for all }u' \in u(t)+ H^1_0(U). 
\end{equation}
\item (\emph{Energy dissipation inequality}) For every $0 \leq t_0 \leq t_1 \leq T$ it holds
\begin{equation}\label{e.diss-inequality}
 \mathcal{J}(u({t_0}))-\mathcal{J}(u({t_1}))  + \int_{t_0}^{t_1}2 \dot{F}(t)P(t) \ dt \geq \Diss(u(t_0),u(t_1)).
\end{equation}
   Here $P(t) = P(u(t))= \int_{\partial U} \frac{\partial u(t)}{\partial n} \ dS$ is an associated pressure. 
  \end{enumerate}
  \end{definition}

A formal first variation of the global stability condition \eqref{e.stability} shows that $u(t)$ should indeed be a solution of \eqref{e.stability-condition-intro}. However, due to the global nature of the stability condition, the solution can in principle jump to a more favorable energy configuration earlier than would be physically reasonable; see \sref{example-jump-conditions} for a discussion of an example. Therefore we do not expect that the obstacle solutions \OQE{} and energy solutions \EE{} coincide when jumps occur.

The authors showed in \cite{FKPi} that the limit points as timestep $\delta \to 0$ of the following minimizing movements scheme are energy solutions.
Define a piecewise constant interpolation 
 \begin{equation}
 \label{e.minimizing-movement-scheme-solution}
 u_{\delta}(t):= u_{\delta}^k \ \hbox{ and } \ F_\delta(t) = F(k\delta) \ \hbox{ if } \ t\in [k\delta, (k+1)\delta).
 \end{equation}
of the time-discrete scheme
\begin{equation}\label{e.minimizing-movement-scheme}
u_{\delta}^k \in \mathop{\textup{argmin}} \left\{ \mathcal{J}(w) + \Diss(u^{k-1}_\delta, w): w\in F(k\delta)+  H^1_0(U)\right\}.
\end{equation}

As our second main result, we show that the obstacle solutions coincide with the notion of energy solutions that is generated from the minimizing scheme in the case of strongly star-shaped data. As we will see convexity is not preserved in the problem, \sref{ex-non-convexity}, and thus star-shapedness is the most natural context for this problem where no topology change or jump occurs in the problem, which allows pointwise regularity analysis.

\begin{theorem}[see \tref{equivalence} and \tref{oqe-MM-relation}]\label{t.main-1}  Let $F$ satisfy \eqref{e.forcing}, and let $u(0)$ be strongly star-shaped with bounded support $\R^d \setminus U \subset \Omega_0$. Further assume that $\partial\Omega_0$ is $C^{1,\alpha}$ and $u(0)$ satisfies \eqref{e.stability-condition-intro}. Then the following holds for $u$,  the unique obstacle solution \OQE{} on $[0,T]$ with initial data $u(0)$:
\begin{enumerate}[label = (\roman*)]
\item\label{part.main1-p3} The unique \OQE{} solution $u$ is also the unique energy solution that is the limit of minimizing movement scheme with initial data $u(0)$. Moreover the solutions \eqref{e.minimizing-movement-scheme-solution} of the discrete-time minimizing scheme converge uniformly to $u(t)$ with a uniform rate that only depends on $F$, $\mu_{\pm}$.

\item\label{part.main1-p1}  $u(t)$ is strongly star-shaped for each time.
\item\label{part.main1-p2} $u(t)$ and  $\Omega(u(t))= \{u(t)>0\}$ is $C^{0,1}_tC_x \cap L^\infty_tC^{1,\min\{\frac{1}{2},\alpha\}-}_x$. 
\end{enumerate}
\end{theorem}

Philosophically, (a) indicates that the local stability and dynamic slope conditions contain all the information in \EE{} or \OQE{} \emph{except for the jump law}. We expect that the spatial regularity $C^{1,\frac{1}{2}}_x$ described in \partref{main1-p2} is optimal.  Specifically, when the contact line ``peels" or ``de-laminates" from the initial data, as shown in \fref{convexity-loss}, free boundary $\partial\Omega(u(t))$ looks like a solution of a thin-obstacle / Signorini problem near its thin free boundary.  See \sref{bernoulli-reg}, specifically \tref{flat-implies-regular-obstacle} and \eref{signorini-expansion}, for a more precise description of this asymptotic expansion.  It is well known that Signorini solutions with smooth obstacle may peel away from their thin free boundary with as little as $C^{1,\frac{1}{2}}_x$ regularity due to the model solution $\textup{Re}((x+iy)^{3/2})$, see the survey \cite{FernandezRealSurvey} for details and references.  Although we do not fully explore the sharp asymptotic expansion of the free boundary near a de-lamination point $x_0 \in \partial \Omega(u(t))$, we do establish the (almost) optimal upper growth bound $|x-x_0|^{1+\min\{1/2, \alpha\}-}$.

Besides the regularity theory developed in \sref{bernoulli-reg}, the novel comparison principle for viscosity solutions of \eqref{e.stability-condition-intro} and \eqref{e.dynamic-slope-condition-intro}, \pref{MVS-OVS-comparison}, also plays a central role in this theorem. The convergence result for minimizing movements solutions part \partref{main1-p3} follows a similar idea to Chambolle's proof of convergence of the Almgren-Taylor-Wang / Luckhaus-Sturzenhecker schemes for mean curvature flow \cite{chambolle04,ATW,LS}.  

The star-shapedness is a key hypothesis in \tref{main-1}. Perhaps most importantly, the star-shaped geometry is used to ensure that there is no jump in the obstacle solution \OQE{}. This is crucial to obtain part \partref{main1-p3}, since the energy solutions and obstacle solutions have different jump laws, as we illustrate by an example in Section~\ref{s.examples}.

The local cone monotonicity implied by star-shapedness is also used to obtain the regularity result in \partref{main1-p2}. Cone monotonicity allows us to show, in \sref{bernoulli-reg}, that all blow-ups at the free boundary are half-plane solutions. Then we can invoke the ``flat means smooth" results of \cite{ChS,FerreriVelichkov} for Bernoulli obstacle problems.  

It is difficult to derive geometric properties of the evolution purely from the energetic structure. At the time of the original appearance of this manuscript we were not aware of any purely energetic method to obtain higher regularity of the free boundary.  Even seemingly simple properties seemed difficult to prove just from energetics, for example it was not clear whether arbitrary energy solutions \EE{} must respect time monotonicity of the Dirichlet forcing.  After the original appearance of this work, Collins and the first author \cite{CollinsFeldman} have partially resolved these questions in the context of minimizing movements solutions.

\subsection{Open questions}\label{s.discussion}

An important motivation for introducing the obstacle solutions, even outside of the star-shaped case, is that it handles time jump discontinuities well. This is in contrast to the global energetic solutions studied in \cite{alberti2011,FKPi} which jump as early as is energetically favorable. Instead, the obstacle solution jump ``as late and as little as possible". This is regarded as a more physically accurate jump condition, similar to the notion of balanced viscosity solutions: see for example \cite{MielkeRossiSavare1} or \cite[Chapter 3.8.2]{mielke2015book}.  The obstacle solution dissipates the ``right" amount of energy on its jumps, but does not obviously yield an energetic notion of jump dissipation. It would be interesting to study the possible connection of the balanced viscosity notion with our obstacle solution \OQE{}.

 It would be interesting to study the regularity of obstacle solutions outside of the star-shaped setting. Unlike the energy solutions considered in \cite{CollinsFeldman}, obstacle solutions certainly develop singularities even in low dimensions.  For example when two components touch and then merge after advancing one expects at least one singular point with a blow-up profile $(1+\mu_+)|x_d|_+$.  This does not occur for energy solutions since the jump discontinuity occurs before the components touch.  Unfortunately the main tool in the general regularity theory, the Weiss monotonicity formula \cite{Weiss}, applies to variational solutions. It is not clear if obstacle minimal supersolutions / maximal subsolutions satisfy the Weiss monotonicity formula.

\subsection{Notations and conventions} We list several notations and conventions which will be in force through the paper.

\begin{enumerate}[label = $\vartriangleright$]
\item We call a constant \emph{universal} if it only depends on $d$ and $\mu_+>0$, $\mu_- \in (0,1)$.
\item We will refer to universal constants by $C \geq1$ and $0<c \leq 1$ and allow such constants to change from line to line of the computation.  
\item We often abuse notation and write $\Omega(t)$ instead of $\Omega(u(t))$ etc.
\item $u^*$ and $u_*$ denote the upper-semicontinuous envelope and the lower-semicontinuous envelope of $u$, respectively, see \eqref{envelopes}. We write USC and LSC respectively as shorthand for upper semicontinuous and lower semicontinuous.
\item $F + H^1_0(U)$ refers to the space of functions in $H^1(U)$ with trace $F$ on $\partial U$.
\end{enumerate}

\subsection*{Acknowledgments} W. Feldman was partially supported by the NSF grants DMS-2009286 and DMS-2407235. I. Kim was partially supported by the NSF grant DMS-2153254. N. Pozar was partially supported by JSPS KAKENHI Kiban C Grant No. 23K03212. 

\section{Motivating examples} 
\label{s.examples}

In order to introduce the problem and motivate the phenomena we will study in the paper we present several examples with analytical computations and numerical simulations.  

\subsection{Numerical simulations}

The obstacle solutions \OQE\ can be relatively easily approximated by a large time limit ($\tau \to \infty$) of a dynamic contact angle problem: Suppose that $F$ is increasing on $[s, t]$. Given $u(s)$, $u(t)$ can be found as the limit $\tau \to \infty$ of the unique, monotone solution of the free boundary problem
\begin{align*}
\left\{\begin{aligned}
- \Delta w(\tau) &= 0, &&\text{in } \{w(\tau) > 0\} \cap U,\\
V_n &= \max(|\nabla w| - (1 + \mu_+)^{1/2}, 0) && \text{on } \partial\{w(\tau) > 0\} \cap U,\\
w(\tau) &= F(t), && \text{on } \partial U,\\
\{w(0) > 0\} &= \{u(s) > 0\}.
\end{aligned}\right.
\end{align*}
 Analogously, for $F$ decreasing on $[s, t]$ we replace the velocity law by $V_n = \min(|\nabla w| - (1 - \mu_-)^{1/2}, 0)$.

A numerical solution can be found by adapting the level set method introduced in \cite{Gibou02}, stopping at $\tau$ when $|V_n| < \e$ for some small parameter $\e > 0$. Plots in \fref{convexity-loss} were produced this way.

\subsection{Jump conditions for (global) energetic vs obstacle solutions}
\label{s.example-jump-conditions}
Next we consider an example where the jump time for the energy solutions is different from the jump time for the obstacle evolution.  Heuristically speaking the obstacle evolution solutions jump as late and as little as possible, while the (global) energy solutions jump whenever it becomes energetically favorable.

Consider a domain $U$ which is the complement of two disjoint closed disks and initial data given by two disjoint annuli.  Then under increasing $F(t)$ the solution will consist of two disjoint annuli until the value of the forcing when the boundaries of the two annuli meet at a single point.  As $F$ continues increasing past that critical value the solution will need to jump outwards to a new state.  This situation is depicted in the simulations in \fref{two-discs}.

The energy solution with the same data and forcing must jump before the value of the forcing when the two annuli touch.  One way to see this is that the blow-up at the touching point of the annuli is a two plane solution of the form
\[v(x) = (1+\mu_+)^{1/2} |x \cdot e|.\]
However this blow-up does not satisfy the global stability condition in \EE{}, since the harmonic replacement in any open region has the same positivity set but lower Dirichlet energy.  Since any blow-up of a globally stable state is also globally stable we conclude that no energy solution can coincide with the obstacle solution all the way to the jump time.

\subsection{Convexity is not preserved under the obstacle evolution}\label{s.ex-non-convexity} Consider the Bernoulli free boundary problem in the complement of a convex obstacle $K$
\[- \Delta u = 0 \ \hbox{ in } \ \{u>0\} \setminus K, \ \hbox{ with } \ |\grad u| =1 \ \hbox{ on } \ \partial \{u>0\} \setminus K.\]
It is known, see Henrot and Shahgholian \cite{HenrotShahgholian}, that the (unique) compactly supported solution of this problem has convex super-level sets.  In particular if one considers the obstacle solution \OQE{} without pinning $\mu_\pm = 0$, then the solution $u(t)$ of the obstacle evolution (which depends only on the current value $F(t)$ due to the lack of hysteresis) is convex at all times.

It is natural to ask whether convexity is still preserved under \OQE{} in the case of nontrivial pinning interval $\mu_\pm >0$.  In fact it is not. We give a simulation of a counterexample in \fref{convexity-loss} and present a sketched proof here.

Consider the case of $K = B_1$, $F(0) = 1$, and an initial region $\Omega_0$ which is a ``stadium" type initial data which is a large portion of the strip $-a < x_2 < a$ capped off by two circles of radius $a$ centered at $(\pm b,0)$.  Here $0 < a-1 \ll 1$ and $b \gg 1$.  The initial region $\Omega_0$ and $K$ uniquely determine the initial profile $u_0$, which does have convex super-level sets due to \cite{CaffarelliSpruck}.  Fix the pinning interval by the relations
\[\mu_+ := \max_{\partial \Omega_0} (|\grad u_0|^2-1) \ \hbox{ and } \ \mu_-:=\max_{\partial \Omega_0} (1 - |\grad u_0|^2). \]
By symmetry considerations the first maximum is achieved at $(0,\pm a)$ and the second at $(\pm (b+a),0)$.  We can guarantee that both maxima are positive by choosing the $a$ close to $1$ and $b$ large.  Now increasing $F(t)$ slightly above $1$ the free boundary needs to immediately move outwards near $(0,\pm a)$ since the slope there is already saturated at $|\grad u_0(0,a)|^2 = 1+\mu_+$, but since this is exactly the maximum value of the slope the outwards movement will only be in a small neighborhood of those two points. This motion must produce nonconvexity because the domain $\Omega(t)$ will have outward normal $e_2$ both at some point $(0,a+\delta(t))$ and at $(b,a)$.

This argument could be justified rigorously using \tref{main-1}.

\begin{figure}
\begin{tabular}{ll}
 \includegraphics[width=0.5\textwidth]{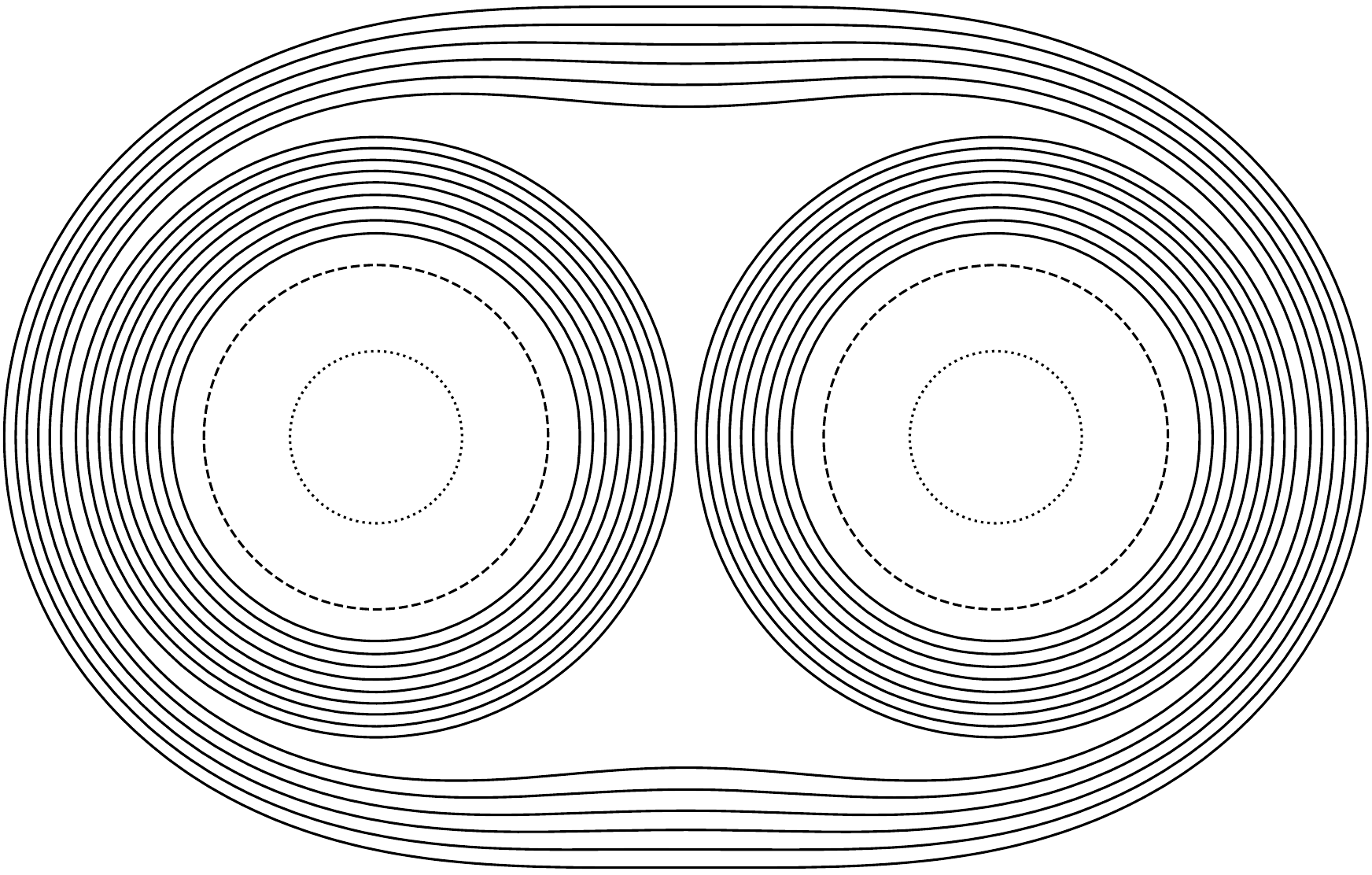}
 &
 \includegraphics[width=0.5\textwidth]{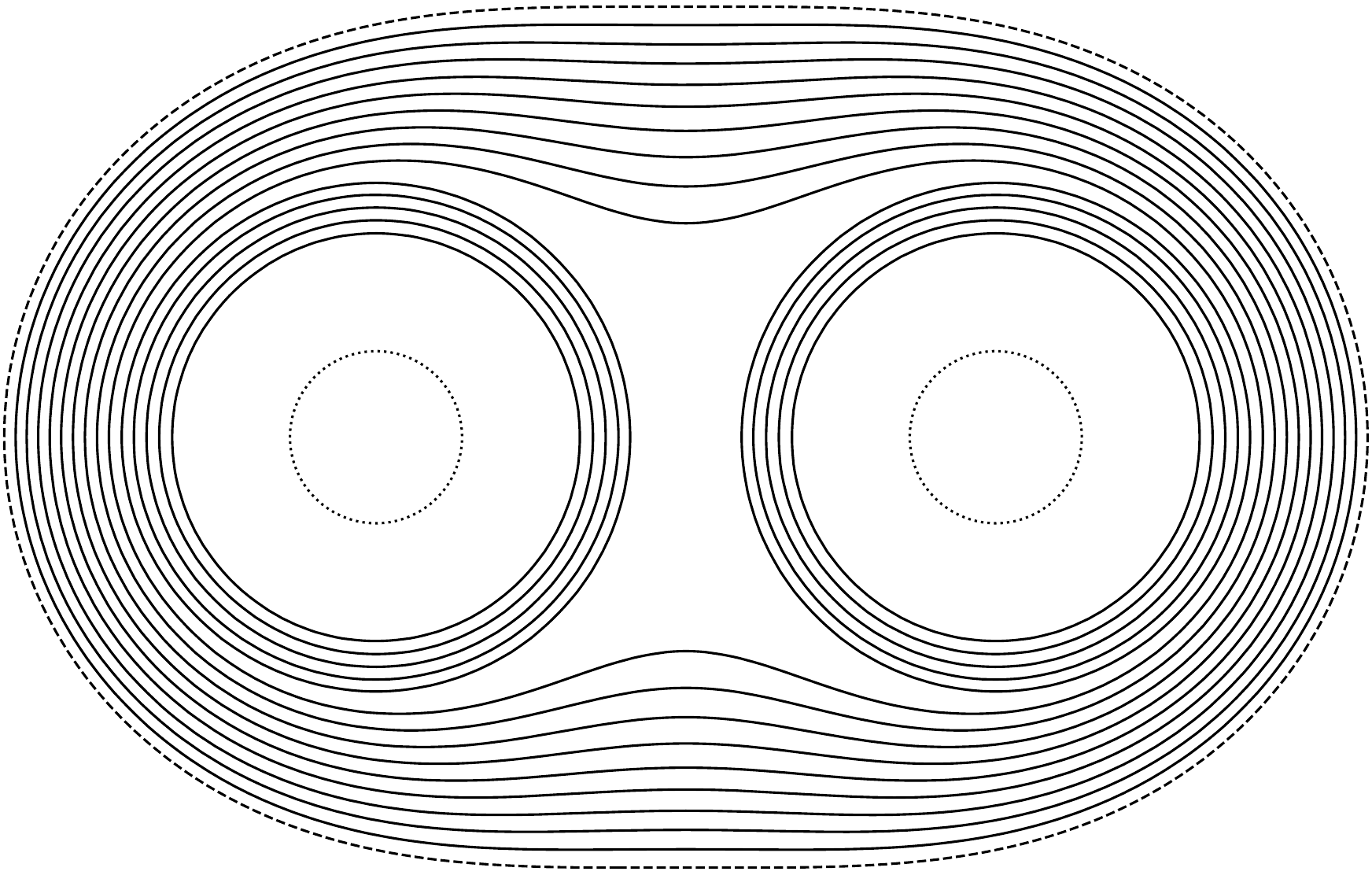}
 \\
 \includegraphics[width=.5\textwidth]{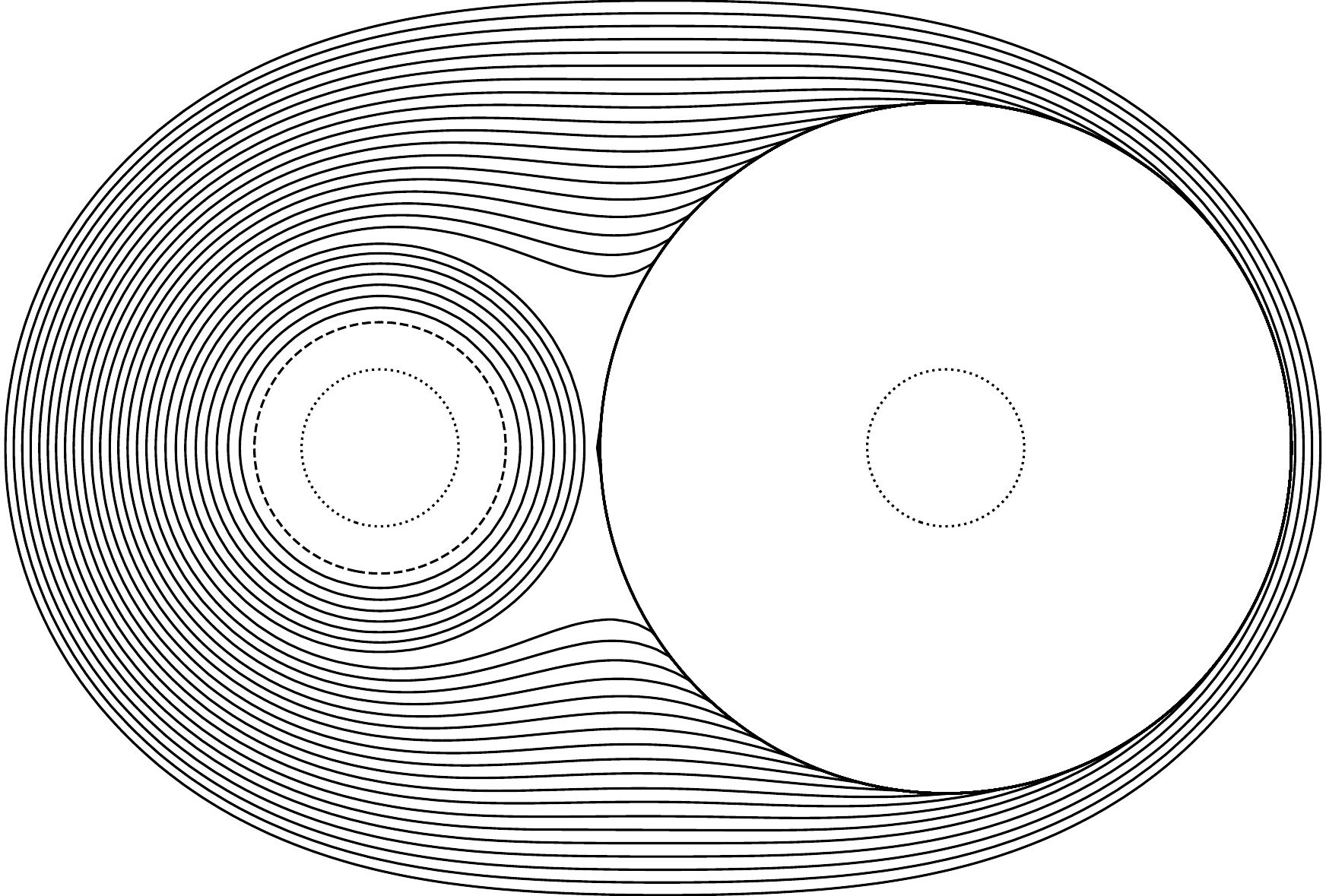}
 &
 \includegraphics[width=.5\textwidth]{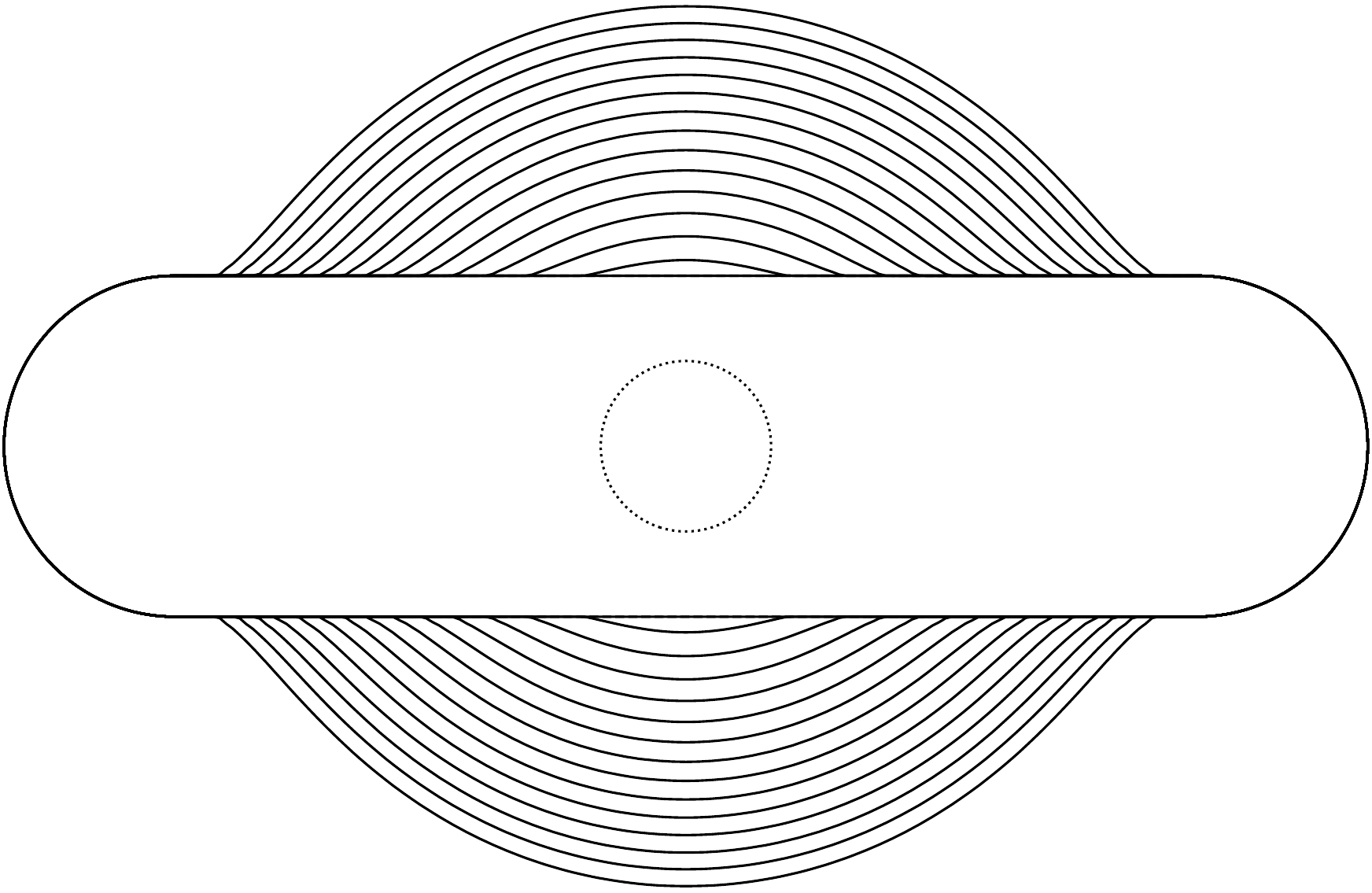}
\end{tabular}
\caption{Plots of boundaries of obstacle solution \OQE{} simulations. Solid curves represent $\partial \Omega(t) \cap U$ plotted for evenly spaced values of $F(t)$, with the initial shape dashed. $\partial U$ is given by a dotted curve. Top left: Disconnected annuli initial data, jump discontinuity on touching. 
Top right: Receding situation (decreasing $F(t)$) with initial data given by the last step of the top left image. Note that the jump occurs at a different configuration, as late as possible. 
Bottom left: Different radius annuli, free boundary peels from the larger annulus after the jump.  Bottom right: Stadium type initial data, convexity is not preserved.}
\label{f.small-large-merge}
\label{f.two-discs}
\label{f.convexity-loss}
\end{figure}

\section{Regularity theory of Bernoulli obstacle problems} 
\label{s.bernoulli-reg}
In this section we consider a pair of obstacle problems for the Bernoulli free boundary problem, one with an obstacle from above and the other with an obstacle from below.  The two problems are similar but not exactly symmetric, as we will see below in the analysis.  The regularity theory of the problem with obstacle from above has been developed by Chang-Lara and Savin \cite{ChS}, and as we were finishing preparing this paper their theory has been extended to the obstacle from below case by Ferreri and Velichkov \cite{FerreriVelichkov}. We present the problems, recall the flat implies smooth regularity results from the above-mentioned works, and then show how to achieve the initial flatness and full regularity under a cone monotonicity hypothesis which is appropriate for our work.

\subsection{Bernoulli obstacle problems} Let $U$ be an open region, the domain.  We say that $u$ is a solution / supersolution / subsolution of the (unconstrained) Bernoulli free boundary problem in $U$ if
\begin{equation}\label{e.bernoulli-unconstrained}
\begin{cases}
\Delta u = 0 &\hbox{in } \{u>0\} \cap U,\\
|\grad u| = 1 &\hbox{on } \ \partial \{u>0\}  \cap U.
\end{cases}
\end{equation}
Let the \emph{obstacle} $O$ be another open region with $C^{1,\alpha}$ boundary.

\begin{definition}
A function  $u \in C(\overline{U})$ is a solution of the Bernoulli problem in $U$ with obstacle $O$ from below if
\begin{equation}\label{e.obstacle-below}
\begin{cases}
\Delta u = 0 &\hbox{in } \{u>0\} \cap U\\
u>0 &\hbox{in }  O \cap U \\
|\grad u| = 1 &\hbox{on } \ (\partial \{u>0\}  \setminus \overline{O}) \\
|\grad u| \leq 1 &\hbox{on } \Lambda:=\ (\partial\{u>0\} \cap \partial O).
\end{cases}
\end{equation}
\end{definition}
\begin{definition}
A function $u \in C(\overline{U})$ is a solution of the Bernoulli problem in $U$ with obstacle $O$ from above if
\begin{equation}\label{e.obstacle-above}
\begin{cases}
\Delta u = 0 &\hbox{in } \{u>0\} \cap U\\
u=0 &\hbox{in } \ \overline{U} \setminus O\\
|\grad u| = 1&\hbox{on } \ (\partial \{u>0\}  \cap O) \\
|\grad u| \geq 1 &\hbox{on } \Lambda:= \ (\partial\{u>0\} \cap \partial O).
\end{cases}
\end{equation}
\end{definition}
See \fref{obstacle-pics} for depictions of obstacle solutions. The PDEs are solved in the standard viscosity sense, see for example \cite{ChS} or \dref{test-fcns-vs-def} below.

\begin{figure}
\begin{minipage}{.48\textwidth}
\begin{tikzpicture}[scale = 1.1]

\def\opa{.6}
\begin{scope}[xscale = 2, rotate=60]


\filldraw[gray!10] (0,0) -- (2,0) -- (2,1) -- (0,1);

\filldraw[gray] (2,1) -- (0,1) -- (0,1.7) -- (2,1.7) -- cycle;
\draw[black,dashed, very thick] (2,1) -- (0,1);

\begin{scope}[opacity = \opa]
\filldraw[gray!40] (0,1) -- (.75,1) .. controls ($(.75,1)+(.5,0)$) .. (2,1.5) -- (3,2.5) .. controls ($(3,2.5)$) and ($(1.75,2)+(.5,0)$) .. (1.75,2) -- (1,2) -- cycle;
\end{scope}
\foreach \y in {1, 1.2, ..., 1.8}
\draw[black,thin] (\y-1-.1,\y) -- (\y-1+.75,\y) .. controls ($(\y-1+.75,\y)+(.5,0)$) .. (2+\y-1,\y+.5);

\node at (1,.5) {$\{u=0\}$};
\node[right] at (2.4,1.9) {$|\grad u| = 1$};
\node[below] at (0,1) {$|\grad u| \geq 1$};
\node at (.25,1.5) {$O$};

\end{scope}

\end{tikzpicture}
\end{minipage}
\begin{minipage}{.48\textwidth}
\begin{tikzpicture}[scale = 1.1]
\def\opa{.6}

\begin{scope}[xscale = 2, rotate=60]
\def\yA{1}
\filldraw[gray!10] (0,0) -- (2,0) -- (2+\yA-1,\yA) -- (2+\yA-1-.75,\yA) .. controls ($(2+\yA-1-.75,\yA)+(-.5,0)$) .. (\yA-1,\yA-.5);

\filldraw[gray] (2,1) -- (0,1) -- (0,1.7) -- (2,1.7) -- cycle;
\draw[black,dashed, thick] (2,1) -- (0,1);

\begin{scope}[opacity=\opa]
\def\yB{2}
\filldraw[gray!40] (2+\yA-1,\yA) -- (2+\yA-1-.75,\yA) .. controls ($(2+\yA-1-.75,\yA)+(-.5,0)$) .. (\yA-1,\yA-.5) -- (\yB-1,\yB-.5) .. controls ($(\yB-1,\yB-.5)+(.1,.2)$) and ($(2+\yB-1-.75,\yB)+(-.5,0)$) .. (2+\yB-1-.75,\yB) -- (2+\yB-1,\yB) -- cycle;
\end{scope}
\foreach \y in {1, 1.2, ..., 1.8}
\draw[black,thin]  (2+\y-1+.1,\y) -- (2+\y-1-.75,\y) .. controls ($(2+\y-1-.75,\y)+(-.5,0)$) .. (\y-1,\y-.5);

\node at (1,.5) {$\{u=0\}$};
\node[right] at (2.4,1.4) {$|\grad u| \leq 1$};
\node[below] at (0,.5) {$|\grad u| = 1$};
\node at (.25,1.5) {$O$};

\end{scope}

\end{tikzpicture}
\end{minipage}
\caption{Left: Obstacle from above, slope is larger than $1$ everywhere and saturates where free boundary bends into $O$. Right: Obstacle from below, slope is smaller than $1$ everywhere and saturates where free boundary bends away from $\overline{O}$.}
\label{f.obstacle-pics}
\end{figure}
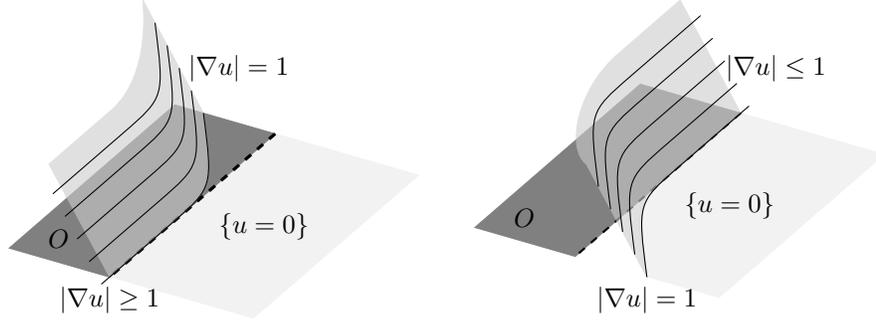

\subsection{Additional hypotheses} The obstacle below (resp. above) problems do not provide any a-priori bound on the slope at the free boundary from below (resp. above). Given a regular domain $O$ and a specific boundary data on $\partial U$ one could establish such bounds on the interior.  However it is more convenient for us to just list these additional bounds as hypotheses which will be in force for some (but not all) of the statements below.  Let $0 < \kappa < 1$ and consider the hypothesis: in the obstacle from below case
\begin{equation}\label{e.obstacle-below-hyp}
|\grad u| \geq \kappa \ \hbox{ in the viscosity sense on } \ \partial \{u>0\} \cap U,
\end{equation}
and in the obstacle from above case
\begin{equation}\label{e.obstacle-above-hyp}
|\grad u| \leq \kappa^{-1} \ \hbox{ in the viscosity sense  on } \ \partial \{u>0\} \cap U.
\end{equation}

\subsection{Existence and typical examples} An indicative example of $u$ solving a Bernoulli problem with obstacle from below  \eqref{e.obstacle-below} is the Perron's method obstacle minimal supersolution 
\begin{equation}\label{e.minimal-example}
u(x) := \inf\{v(x): \hbox{$v$ is a supersolution of \eref{bernoulli-unconstrained}, $v = g$ on $\partial U$, and } v>0 \hbox{ in } O \cap \overline{U}\}.
\end{equation}
Here $g \in C(\partial U)$ is some boundary condition.  In order for the minimal supersolution to also be positive on $O$ we need to put some additional condition on $g$. Notice that every supersolution in \eref{minimal-example} is above $w$, the solution of the Dirichlet problem
\[
\Delta w = 0 \ \hbox{ in } \ O \cap U, \
w = 0 \ \hbox{ on } \ \partial O \cap U, \ \hbox{ and } \ 
w = g  \ \hbox{ on }\  \partial U \cap \overline{O}.
\]
So if $w >0$ in $O$ then $u$ will be positive in $O$ as well.  In particular it would suffice to assume that
\[g > 0 \ \hbox{ on } \ O \cap \partial U.\]
The obstacle solution property follows from the typical Perron's method arguments. Local upward perturbations are possible everywhere in $U$, so $u$ is a supersolution everywhere, while local downward perturbations are possible only away from $\partial \{u>0\} \cap \overline{O}$ limiting the subsolution property.

An indicative example of $u$ solving a Bernoulli problem with obstacle from above \eref{obstacle-above} is the obstacle maximal subsolution 
\begin{equation}\label{e.maximal-example}
u(x) := \sup\{v(x): \hbox{$v$ a subsolution of \eref{bernoulli-unconstrained} in $U$, $v = g$ on $\partial U$, and } \{v>0\} \subset O\}.
\end{equation}
For consistency the positivity set of $g$ on $\partial U$ should be contained in $O \cap \partial U$.  As before Perron's method applies, arbitrary downward perturbations are possible so $u$ is a subsolution everywhere, but local upward perturbations are only allowed away from $\partial \{u>0\} \cap \overline{O}$.   

\begin{remark}
As in \cite{ChS} it is also possible to construct solutions to \eref{obstacle-above} by minimal supersolution Perron's method, or by energy minimization.  However, it is the maximal subsolutions that we will encounter in this work.  This is something we need to be careful with, since non-degeneracy at the free boundary is a more delicate issue for maximal subsolutions.
\end{remark}

\subsection{Flat implies $C^{1,\frac{1}{2}-}$ and regularity of cone monotone solutions} 

First let us recall the flat implies $C^{1,1-}$ regularity away from the obstacle.  This result is originally due to Caffarelli \cite{Caffarelli1,Caffarelli2}. A more recent alternative proof by De Silva \cite{DeSilva} has motivated the techniques used to study the obstacle problems we are considering.

\begin{theorem}\label{t.flat-implies-regular-unconstrained}(Flat implies $C^{1,1-}$ for unconstrained Bernoulli \cite{Caffarelli1,Caffarelli2})
Let $u$ solve \eref{bernoulli-unconstrained} in $B_1$ and $0 \in \partial \{u>0\}$.  For any $\beta \in (0,1)$ there is $\ep_0>0$ and $C \geq 1$ universal so that if
\[(x \cdot e - \ep)_+ \leq u(x) \leq (x\cdot e + \ep)_+ \ \hbox{ in } \ B_1\]
for some $\ep \leq \ep_0$, then for all $0 < r < 1$ and some $|\grad u(0)| = 1$
\[(x \cdot \grad u(0) - C\ep r^{1+\beta})_+ \leq u(x) \leq (x\cdot \grad u(0) + C\ep r^{1+\beta})_+ \ \hbox{ in } \ B_r.\]
\end{theorem}

Next we recall a similar flat implies smooth result on the contact set between the solution and the obstacle. Consider the contact set of the free boundary with the obstacle $\Lambda := \partial \{u>0\} \cap \partial O$.  Let us denote $\partial'\Lambda$ to be the boundary of $\Lambda$ relative to $\partial O$. 

\begin{theorem}\label{t.flat-implies-regular-obstacle}(Flat implies $C^{1,\frac{1}{2}-}$ at the contact set \cite{ChS,FerreriVelichkov})
Let $u$ solve either \eref{obstacle-above} or \eref{obstacle-below} in $B_1$ and $0 \in \partial' \Lambda$, $O$ be a $C^{1,\alpha}$ obstacle, and $e$ be the inward normal to $O$ at $0$.  For any $0 < \beta < \min\{\alpha,\frac{1}{2}\}$ there is $\ep_0>0$ and $C \geq 1$ universal so that if
\[(x \cdot e - \ep)_+ \leq u(x) \leq (x\cdot e + \ep)_+ \ \hbox{ in } \ B_1\]
for some $\ep \leq \ep_0$, then for all $0 < r < 1$ 
\[(x \cdot e - C\ep r^{1+\beta})_+ \leq u(x) \leq (x\cdot e + C\ep r^{1+\beta})_+ \ \hbox{ in } \ B_r.\]
\end{theorem}

In both the obstacle above and obstacle below cases the proof is based on establishing an asymptotic expansion for flat solutions of the form
\begin{equation}\label{e.signorini-expansion}
    u(x) = (x \cdot  e + \ep w(x) + o(\ep))_+
\end{equation}
where $w$ is a solution of a certain Signorini or thin obstacle problem and $\ep$ is the flatness in $B_1$ as in the hypothesis of \tref{flat-implies-regular-obstacle}.  The $C^{1,\frac{1}{2}}$ optimal regularity for the Signorini problem is the reason for the $C^{1,\frac{1}{2}-}$ regularity which appears here, and is likely (almost) optimal.  The idea of the argument, which is based on a compactness principle using a special Harnack inequality for flat solutions, goes back to the work of De Silva \cite{DeSilva}.

It is important for us to obtain a full regularity result without the flatness assumption.  Of course some assumption is still necessary, singular solutions exist in higher dimensions, and fitting with the strongly star-shaped geometry we study later in the paper we will consider obstacle solutions satisfying a cone monotonicity condition.  For the statement define
\begin{equation}
\textup{Cone}(e, \theta):= \{e': e'\cdot e \geq 1-\theta\} \ \hbox{ for } \ e \in S^{d-1} \ \hbox{ and } \ \theta \in (0,1).
\end{equation}

\begin{theorem}
\label{t.bernoulli-reg}
 Suppose $O$ is a bounded set with $C^{1,\alpha}$ boundary, and that $u$ solves either \eref{obstacle-above} and \eref{obstacle-above-hyp} or \eref{obstacle-below} and \eref{obstacle-below-hyp} in $B_1$. Assume also that there are $e \in S^{d-1}$ and $1\geq c_0>0$ so that $u$ is monotone increasing in the directions of the cone $\textup{Cone}(e, c_0)$.
 
 Then for any $0 < \beta < \min\{\alpha,\frac{1}{2}\}$ the positivity set $\{u>0\}$ is a $C^{1,\beta}$ domain in $B_{1/2}$ and $u\in C^{1,\beta}(\overline{\{u>0\}} \cap B_{1/2})$.  The $C^{1,\beta}$ norms depend only on $\alpha$, $\kappa$, $c_0$, $d$, and the $C^{1,\alpha}$ norm of $\partial O$. 
\end{theorem}

\begin{remark}\label{r.signorini-optimal}
Chang-Lara and Savin \cite{ChS} prove the \emph{optimal} $C^{1,\frac{1}{2}}$ regularity in the case when the obstacle is $C^{1,\alpha}$ regular for some $1/2 < \alpha \leq 1$. This matches the regularity of the thin obstacle/Signorini problem, which appears as the first order term in the asymptotic expansion for $\ep$-flat solutions in the flatness parameter $\ep$.  Since our obstacles are generated from the evolution itself and therefore have at most $C^{1,\frac{1}{2}}$ regularity after the first monotonicity change, obtaining the optimal regularity is more delicate in our problem.  As shown in \cite{RulandShi} there can be logarithmic losses of regularity at this critical scaling. To avoid getting into these technical details we aim to establish the (almost) optimal regularity $C^{1,\frac{1}{2}-}$.
\end{remark}

We need to establish the initial flatness in order to apply \tref{flat-implies-regular-unconstrained} and/or \tref{flat-implies-regular-obstacle} \cite{ChS}. This is a bit different in our case from the way that initial flatness is established in \cite{ChS} (minimal supersolutions) or \cite{FerreriVelichkov} (one-sided energy minimizers).

The first difference is in the obstacle from below case. This issue appears in \lref{contact-blowup}, where we establish the initial flatness of $u$ at points of $\partial'\Lambda$ based on the existence of a non-tangential slope. In the case of the obstacle from above \eref{obstacle-above} the positive set of $u$ lies inside of $O$ which has smooth boundary. This means that a non-tangential gradient of $u$ at the free boundary point in contact with $\partial O$ can describe the leading-order behavior of $u$ near the point in the entire positive set.  This is no longer true with \eref{obstacle-below}, where the positive set contains $O$. This motivates the cone monotonicity hypothesis on $u$ in the obstacle from below case.  

The other difference is in the obstacle from above case.  Unlike \cite{ChS}, who consider minimal supersolutions and energy minimizers, we are interested to study maximal subsolutions.  Non-degeneracy at the free boundary is not known in general for maximal subsolutions.  However, in the case when the free boundary is a Lipschitz graph, in particular in the cone monotone case, non-degeneracy does hold for all viscosity solutions including the maximal subsolution, see \lref{lip-bdry-nondegen} below.  This motivates the cone monotonicity hypothesis on $u$ in the obstacle from above case. 

In both cases the cone monotonicity hypothesis is probably overly strong, however it suffices for the applications in this paper.  The possibility of a more general regularity result is an open question.

We also remark that Lipschitz regularity is easier than non-degeneracy and just follows from the viscosity supersolution property \eref{obstacle-above-hyp}, see \lref{one-phase-lipschitz} below.

\subsection{Initial free boundary flatness}

 In order to obtain the initial flatness we will show the blow-up limit at a free boundary point on the contact set $\Lambda$ exists and is a half-planar supersolution of the unconstrained problem \eref{bernoulli-unconstrained}.

\begin{lemma}\label{l.contact-blowup}
Assume that $u$ is a solution of \eref{obstacle-below} and \eref{obstacle-below-hyp} (or \eref{obstacle-above} and \eref{obstacle-above-hyp}) in $B_1$ which is monotone with respect to the directions of a cone $\textup{Cone}(e,c_0)$, then
 \begin{equation}\label{e.loc-unif-convergence-blowup}
 \lim_{r \to 0} \frac{u(rx)}{r} = (\grad u(0) \cdot x)_+ \ \hbox{ locally uniformly in $\R^d$}
 \end{equation}
 with some gradient $c(d,\kappa)< |\grad u(0)| \leq 1$ (resp. $1 \leq |\grad u(0)|\leq C(d,\kappa)$).  In particular, for any $\e>0$ there exists $r_0>0$ sufficiently small (depending on $u$) so that for any $r \leq r_0$
\begin{equation}\label{e.initial-flatness1}
 |\grad u|(0)(x_d - \ep r)_+  \leq u(x) \leq |\grad u|(0) (x_d + \e r)_+ \ \hbox{ in } \ B_{r}. 
\end{equation}
 \end{lemma}

 \begin{proof} We will just consider the obstacle from below \eref{obstacle-below}, the obstacle above case is in \cite[Lemma 2.6]{ChS}. The idea is that the interior ball condition furnished by the regular obstacle gives a non-tangential blow-up limit and then the cone monotonicity upgrades this to a local uniform limit on the whole space.
 
 Note that $u$ is harmonic in its positivity set, $\{u>0\} \supset O$ in $B_1$ and $O$ is a half-space in $B_1$.  Therefore $0$ is an inner regular point so there is a slope $\grad u(0)$, parallel to the inward normal $e_d$ to $O$ at $0$, such that
 \[
 u(x) = (\grad u(0) \cdot x)_+ + o(|x|) \hbox{ as } x\to 0 \hbox{ non-tangentially in $O$. }
 \]
 See \cite[Lemma 11.17]{CaffarelliSalsa} for the proof, which does extend to the case of H\"older continuous coefficients. By non-degeneracy, \lref{lip-bdry-nondegen}, and the Lipschitz estimate, standard recalled in \lref{one-phase-lipschitz} below, $|\grad u(0)|$ is bounded from below away from zero, and from above.

Here is where we need the cone monotonicity property. The non-tangential information, by itself, does not give us any control in $B_1 \setminus O$.  

From the non-tangential limit, for any $\e>0$, there is $r_0>0$ such that 
\[ u(x) \leq 2|\grad u(0)|\ep r  \ \hbox{ on } \ \{x_d = \ep r\} \cap B_r \ \hbox{ for all } \ r \leq r_0.\]
Since $u$ is monotone increasing in the $e_d$ direction also
\[ u(x) \leq 2|\grad u(0)|\ep r  \ \hbox{ on } \ \{x_d\leq \ep r\} \cap B_r \ \hbox{ for all } \ r \leq r_0. \] 
See \fref{cone-mon-app}. 

 \begin{figure}[t]
  \begin{tikzpicture}[scale = 1.5]

\filldraw[white] (-1,0) -- (1,0) -- (1,-1) -- (-1,-1);
\filldraw[gray!30!white] (-1,0) -- (1,0) -- (1,1) -- (-1,1);
\draw[thick, dashed] (-1,0)--(1,0);
\draw (1,0) -- (0,0) .. controls (-.5,0) and ($(-1,-.3)+ (.1,.1)$) .. (-1,-.3);
\draw[dotted] (-1,.3) -- (0,0) -- (1,.3);

\filldraw[black] (0,0) circle (.025);

\node[right] at (1,0) {$x_n = 0$};
\node[below] at (0,0) {$0$};
\node[below] at (-1,-.3) {$\partial \{u>0\}$};
\node at (0,.5) {$O$};

\end{tikzpicture}
  \caption{Asymptotic expansion in non-tangential cone plus monotonicity also gives control in $\{x_n \leq 0\}$.}
  \label{f.cone-mon-app}
  \end{figure}
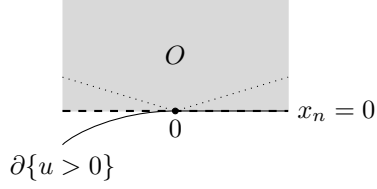

  Thus $ \lim_{r \to 0} \frac{u(rx)}{r} = 0$ for $x_d \leq 0$ and we have upgraded the non-tangential limit to local uniform convergence of the blow-up sequence \eref{loc-unif-convergence-blowup}. Then the uniform stability of viscosity solutions implies that $|\grad u(0)| \leq 1$.
  
  Using non-degeneracy again we can prove the convergence of the free boundaries as well \eref{initial-flatness1}.
  
  \end{proof}

 If we use \lref{contact-blowup} as stated then the initial flatness radius $r_0$ would depend on $u$ is a non-universal way.  In turn the $C^{1,\beta}$ norm of the solution depends on the initial flatness, and would also depend on $u$ in a non-universal way.  Next we show, using the flat implies smooth results, that $r_0$ is actually universal.
 \begin{lemma}
Assume that $u$ is a solution of \eref{obstacle-below} and \eref{obstacle-below-hyp} (or \eref{obstacle-above} and \eref{obstacle-above-hyp}) in $B_1$ which is monotone with respect to the directions of a cone $\textup{Cone}(e,c_0)$. For any $\e>0$ there exists $r_0>0$ sufficiently small, depending only on $\ep$, $c_0$, $\kappa$, $d$ and the $C^{1,\alpha}$ property of $O$, so that for any $r \leq r_0$ there is $\kappa \leq {q} \leq 1$ (resp. $1 \leq {q} \leq \kappa^{-1}$) so that
\begin{equation}\label{e.initial-flatness2}
{q} (x_d - \ep r)_+  \leq u(x) \leq {q} (x_d + \e r)_+ \ \hbox{ in } \ B_{r}. 
\end{equation}
\end{lemma}
\begin{proof}
We just consider the obstacle from below case, the obstacle from above case is similar. The proof is by compactness.  Suppose otherwise, then there is $\ep_0>0$, a sequence of $u_k$ solving \eref{obstacle-below} in $B_1$, a sequence of obstacles $O_k$ all uniformly $C^{1,\alpha}$ regular, and a sequence of radii $r_k \to 0$ so that \eref{initial-flatness2} fails with $r = r_k$ and $\ep = \ep_0$.  By uniform Lipschitz regularity and non-degeneracy we can assume that, on compact subsets of $B_1$, the $u_k$ converge uniformly to some $u$, the $\partial \{u_k>0\}$ converge in Hausdorff distance to $\partial \{u>0\}$, and the $O_k$ and $\partial O_k$ also converge in Hausdorff distance to another $C^{1,\alpha}$ domain $O$. Standard viscosity solution arguments show that $u$ solves \eref{obstacle-below} in $B_1$.  This implies that $u$ is in $C^{1,\beta}(\overline{\{u>0\}} \cap B_{1/2})$ so for any $\ep>0$ there is $r_1>0$ so that \eref{initial-flatness2} holds for $r \leq r_1$.  Let $\ep_1>0$ be sufficiently small so that \tref{flat-implies-regular-unconstrained} and \tref{flat-implies-regular-obstacle} hold, and then $r_1>0$ small so that
\[ |\grad u|(0)(x_d - \tfrac{1}{2}\ep_1 r_1)_+  \leq u(x) \leq |\grad u|(0) (x_d + \tfrac{1}{2}\e_1 r_1)_+ \ \hbox{ in } \ B_{r_1} \]
which implies, for $k$ large enough,
\[ |\grad u|(0)(x_d - \ep_1 r_1)_+  \leq u_k(x) \leq |\grad u|(0) (x_d + \e_1 r_1)_+ \ \hbox{ in } \ B_{r_1}. \]
Then, applying either \tref{flat-implies-regular-unconstrained} or \tref{flat-implies-regular-obstacle} as the case may be, implies that for all $0 < r \leq r_1$
\[ |\grad u_k|(0)(x_d - C\ep_1 r^{1+\beta})_+  \leq u(x) \leq |\grad u_k|(0) (x_d + C\e_1 r^{1+\beta})_+ \ \hbox{ in } \ B_{r} \]
which contradicts the hypothesis for large $k$.
\end{proof}

\section{Local laws: weak solutions, regularity, and comparison principle}\label{s.comparison} 

In this section we study the relationship of the obstacle solutions \OQE{} and solutions of the local stability condition \eqref{e.stability-condition-intro} and the dynamic slope condition \eqref{e.dynamic-slope-condition-intro}, especially in the star-shaped setting.  We establish regularity of obstacle solutions using the results of \sref{bernoulli-reg}.  The central element of a viscosity solutions theory is always the comparison principle.  One of the main results of this section is a comparison principle in the star-shaped setting, \pref{MVS-OVS-comparison}, between an obstacle solution and an arbitrary viscosity solution of \eqref{e.stability-condition-intro} and \eqref{e.dynamic-slope-condition-intro}.

\subsection{Definitions and main results}

Let $F = F(t)$ be a strictly positive $C^1$ function on the time interval $[0,T]$ with a discrete set of critical points, or more precisely,
\begin{align}\label{F-hyp}
F \in C^1([0, T]) \quad \text{and} \quad Z := \{t \in (0, T): F'(t) = 0\} \text{ is discrete}.
\end{align}
In other words, $F$ only changes monotonicity at most finitely many times on any bounded interval.

\subsubsection{Slope condition in the comparison sense} The slope conditions are interpreted in the comparison / viscosity sense.  First we define a notion of sub and superdifferential which is suited to the problem. See also \cite{BardiCapuzzoDolcetta}.

\begin{definition}\label{d.supersubdifferential}
Given a non-negative continuous function $u$ on $U$ define, for each $x_0 \in \partial \Omega(u)\cap U$ 
\begin{align*}
D_+ u(x_0) &= \left\{ p \in  \R^d :  u(x) \leq p \cdot (x-x_0) + o(|x-x_0|) \ \hbox{ in } \  \overline{\Omega(u)}\right\}
\intertext{and}
D_- u(x_0) &= \left\{ p \in  \R^d :  u(x) \geq p \cdot (x-x_0) - o(|x-x_0|)\right\} 
\end{align*}
\end{definition}
This leads to a viscosity notion of slope conditions. This notion is equivalent to the common approach via smooth test functions; see \lref{vs-two-notions-equiv}.
\begin{definition}\label{d.subdifferentials-vs-def}
Suppose $u: U \to [0,\infty)$ is continuous and is subharmonic in $\Omega(u) \cap U$.  Let $G \subset \partial \Omega(u) \cap U$ be a relatively open subset. Then $u$ is a subsolution of 
\[|\grad u|^2 \geq {Q} \ \hbox{ on } G\]
if, for every $x \in G$ and every $p \in D_+u(x)$,
\[|p|^2 \geq {Q}.\]
Similarly, if $u: U \to [0,\infty)$ is continuous and is superharmonic in $\Omega(u) \cap U$ and $G \subset \partial \Omega(u) \cap U$ is a relatively open subset, then $u$ is a supersolution of 
\[|\grad u|^2 \leq {Q} \ \hbox{ on } G\]
if, for every $x \in G$ and every $p \in D_-u(x)$,
\[|p|^2 \leq {Q}.\]
\end{definition}

Finally, let us recall the semicontinuous envelopes $u^*$ and $u_*$ of a function $u: [0, T] \times \overline{U} \to \R$:
\begin{align}\label{envelopes}
u^*(t,x) := \limsup_{(s,y) \to (t,x)} u(s,y),\qquad
u_*(t,x) := \liminf_{(s,y) \to (t,x)} u(s,y),
\end{align}
where $(s,y)$ is always assumed to be in $[0, T] \times \overline U$.
Recall that $u^*$ is USC and $u_*$ is LSC.

We can now precisely define the meaning of the local stability condition \eqref{e.stability-condition-intro} in the viscosity sense.  Note that there is an extra condition, (b) below, on the upper envelope; the asymmetry of the condition will be explained after the definition, see \rref{envelopes-stability} and \rref{envelopes-stability-supersoln} below.

\begin{definition}\label{d.local-stability-visc}
We say that a bounded map $u: [0, T] \to C(\overline U)$ is a \emph{viscosity solution of \eqref{e.stability-condition-intro} on $[0,T] \times U$ in the semicontinuous envelope sense} if 
\begin{enumerate} [label = (\alph*)]
\item\label{part.stab-part-std} $u(t)$ satisfies \eref{stability-condition-intro} for every $t \in [0, T]$.
\item\label{part.local-stab-sub-part} If $p \in D_+ u^*(t_0, x_0)$ for some $x_0 \in \partial\Omega(u^*(t_0))$ then $|p|^2 \geq 1 - \mu_-$. 
\end{enumerate}
\end{definition}
\begin{remark}\label{r.immediate-envelope-properties}
Note that \eref{stability-condition-intro} and uniform boundedness imply that $u(t)$ are uniformly Lipschitz continuous in compact subsets of $U$ by \lref{one-phase-lipschitz}.  Therefore the upper and lower envelopes $u^*(t)$ and $u_*(t)$ are continuous in $x$ for each $t$.  It is also standard that $u^*(t)$ and $u_*(t)$ are respectively subharmonic and superharmonic in their positivity sets.
\end{remark}
\begin{remark}\label{r.envelopes-stability}
Note that if $u: [0, T] \to C(\overline{U})$ such that $u(t)$ is non-degenerate uniformly in $t$ then \partref{local-stab-sub-part} in \dref{local-stability-visc} follow from \partref{stab-part-std}. Indeed, by \lref{vs-two-notions-equiv} we can argue using smooth test functions. Say that $u^*(t) - \phi$ has a strict local maximum at $x_0$ in $\overline{\Omega(u^*(t))}$ and $\Delta \phi(x_0) < 0$. By the uniform nondegeneracy we can deduce that there exists a sequence $t_n \to t$ and $x_n \to x_0$ such that $u(t_n) - \phi$ has a local maximum at $x_n$ in $\overline{\Omega(u(t_n))}$, from which it follows that $u^*(t)$ satisfies \eqref{e.stability-condition-intro}.
\end{remark}
\begin{remark}\label{r.envelopes-stability-supersoln}
A similar argument to \rref{envelopes-stability} can be done for $u_*$ but non-degeneracy is not necessary because the test functions touch from below in $\R^d$ instead of in $\overline{\Omega(u(t))}$. So \partref{stab-part-std} actually directly implies the supersolution analogue to \partref{local-stab-sub-part}: if  $p \in D_- u_*(t_0, x_0)$ for some $x_0 \in \partial\Omega(u_*(t_0))$, then $|p|^2 \leq 1 + \mu_+$.
\end{remark}

\subsubsection{Dynamic slope condition in the comparison sense} In order to formulate the dynamic slope condition \eref{dynamic-slope-condition-intro} we need a weak pointwise sense of positive and negative normal velocity. We will use a notion based on space-time ``light cones" in place of standard barrier functions, see \fref{cone-pics} for a sketch of the geometry. This notion of viscosity solutions is stronger than the usual comparison definition of level set velocity, since we test more free boundary points that have weaker space-time regularity. Nonetheless, in the rate-independent evolution the time variable mostly plays a role of a parameter and so it seems natural to test space and time directions differently.

\begin{figure}
\begin{minipage}{.48\textwidth}
\begin{tikzpicture}[scale = 1]
\filldraw[color = gray!40] (-2,0) -- (0,0) .. controls (-.25,-.25) .. (-2,-1);
\filldraw[pattern=vertical lines] (-1,-1) -- (0,0) -- (1,-1);
\draw[black](0,0) .. controls (-.25,-.25) .. (-2,-1);
\filldraw[black] (0,0) circle (.025);
\draw[black] (-1,-1) -- (0,0) -- (1,-1);
\draw[black,dashed] (-2,0) -- (2,0);

\draw[->] (-2.25,-1) -- (-2.25,0);

\node[right] at (2,0) {$t=t_0$};
\node[left] at (-2.25,-.5) {$t$};
\node at (-1.4,-.3) {$\Omega$};
\node[below] at (0,-1) {$|x-x_0| \leq c (t_0-t)$};
\end{tikzpicture}
\end{minipage}
\begin{minipage}{.48\textwidth}
\begin{tikzpicture}[scale = 1]
\filldraw[color = gray!40] (-2,0) -- (0,0) .. controls (.25,-.25) .. (2,-1) -- (-2,-1);
\filldraw[pattern=vertical lines] (-1,-1) -- (0,0) -- (1,-1);
\draw[black](0,0) .. controls (.25,-.25) .. (2,-1);
\filldraw[black] (0,0) circle (.025);
\draw[black] (-1,-1) -- (0,0) -- (1,-1);
\draw[black,dashed] (-2,0) -- (2,0);

\draw[->] (-2.25,-1) -- (-2.25,0);

\node[right] at (2,0) {$t=t_0$};
\node[left] at (-2.25,-.5) {$t$};
\node at (-1.4,-.3) {$\Omega$};
\node[below] at (0,-1) {$|x-x_0| \leq c (t_0 - t)$};
\end{tikzpicture}
\end{minipage}
\caption{Left: velocity $c$ cone touches $\Omega(t)$ from the outside at $(t_0,x_0)$, interpreted as $V_n(t_0,x_0) \geq c$.  Right: velocity $c$ cone touches $\Omega(t)$ from the inside at $(t_0,x_0)$, interpreted as $V_n(t_0,x_0) \leq - c$. }
\label{f.cone-pics}
\end{figure}
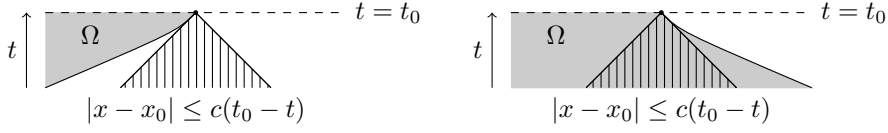
\begin{definition}\label{d.positive-velocity-cone}
\label{d.negative-velocity-cone}
 Given a time varying family of domains $\Omega(t)$, we say that the \emph{outward normal velocity} at $x_0 \in \partial \Omega(t_0)$ is positive and write
\[V(t_0,x_0) > 0\]
if the {\it positive cone condition} holds for some $c > 0$ and $r_0 > 0$, namely
\begin{equation}\label{cone_positive}
\{x: \ |x-x_0| \leq c(t_0-t)\} \subset \Omega(t)^{\complement} \ \hbox{ for } \ \ t_0 - r_0 \leq t < t_0.
\end{equation} 
Similarly, we say that \emph{inward normal velocity} at $x_0 \in \partial \Omega(t_0)$ is negative and write
\[V(t_0,x_0) < 0 \]
if  the {\it negative cone condition} 
\begin{equation}\label{e.cone_negative}
\{x: \ |x-x_0| \leq c(t_0-t)\} \subset \Omega(t) \ \hbox{ for } \ \ t_0 - r_0 \leq t < t_0
\end{equation} 
holds. If we need to clearly specify the domain we write $V(\cdot;\Omega)$.
\end{definition}

Equipped with the notions of nonzero normal velocity and subdifferentials we can now precisely define the meaning of the dynamic slope condition \eqref{e.dynamic-slope-condition-intro} in the viscosity solution sense:

\begin{definition}\label{d.dynamic-slope-condition-visc}
We say that $u: [0, T] \to C(\overline U)$ is a \emph{viscosity solution of \eref{dynamic-slope-condition-intro} on $[0,T] \times U$ in the semicontinuous envelope sense} if
\begin{enumerate} [label = (\alph*)]
\item\label{part.advancing-cond} If $V(t_0,x_0;\Omega(u^*)) >0$ for some $x_0 \in \partial\Omega(u^*(t_0)) \cap U$, then any $p\in D_+u^*(t_0,x_0)$ satisfies $|p|^2 \geq 1+\mu_+.$
 \item\label{part.receding-cond} If  $V(t_0,x_0;\Omega(u_*)) <0$ for some $x_0 \in \partial\Omega(u_*(t_0)) \cap U$, then any $p\in D_-u_*(t_0, x_0)$ satisfies $|p|^2 \leq 1-\mu_-$. 
\end{enumerate}
\end{definition}

Note that since the set $\{V(t_0, \cdot; \Omega(u^*)) > 0\}$ is not necessarily a relatively open subset of the boundary $\partial \Omega(u^*(t_0))$, and analogously for $u_*$, the above definition cannot be easily stated using test functions, \dref{test-fcns-vs-def}.

\subsubsection{Star-shaped comparison principle / equivalence of solution notions} We will show the equivalence of the notions in a strongly star-shaped setting.

\begin{definition}
In the following we say that a set is \emph{strongly star-shaped} if it is star-shaped with respect to all $x$ in a neighborhood of the origin.  A function $u: U \supset\R^d \to \R$ is strongly star-shaped if each of its superlevel sets $\{u > \eta\} \cup U^\complement$ are strongly star-shaped. 
\end{definition}
The important property is that a strongly star-shaped set has a Lipschitz boundary.  This can be seen by drawing the interior and exterior cones created by drawing the rays from every star-shaped center though a point on the boundary, see for example \cite[Figure 2]{FeldmanKim14}.

The main result of this section is the following theorem that asserts that in the strongly star-shaped setting the obstacle solutions are uniquely characterized by the local stability and dynamic slope conditions. 
\begin{theorem}\label{t.equivalence}
If $\R^d \setminus U$ and $\Omega_0$ are both bounded strongly star-shaped sets, $\Omega_0$ is a $C^{1, \alpha}$ domain for some $\alpha > 0$, $F$ satisfies \eqref{F-hyp} and $u(0) \in C(\overline U)$ satisfies \eqref{e.stability-condition-intro}, then     
\[
u \hbox{ is an obstacle solution } \Leftrightarrow u \text{ is a viscosity solution of \eqref{e.stability-condition-intro} and \eqref{e.dynamic-slope-condition-intro},}
\]
in the sense of \dref{local-stability-visc} and \dref{dynamic-slope-condition-visc}.
Furthermore $\Omega(u(t))$ are $C^{1,\beta}$ domains uniformly in $t \in [0,T]$ for some $\beta > 0$.
\end{theorem}

\begin{proof}
The fact that an obstacle solution is also a solution of \eqref{e.stability-condition-intro} and \eqref{e.dynamic-slope-condition-intro} is shown in \lref{OvsMvs}. Then we deduce the regularity in both space and time of obstacle solutions (\lref{space-regularity}, \cref{derivative-time-cont}).  Finally showing that \eqref{e.stability-condition-intro} and \eqref{e.dynamic-slope-condition-intro} imply an obstacle solution follows a comparison-type viscosity argument using sup- and inf-convolutions between a regular solution and a general solution of \eqref{e.stability-condition-intro} and \eqref{e.dynamic-slope-condition-intro} in \pref{MVS-OVS-comparison}.
\end{proof}

\subsection{Obstacle solutions are motion law solutions} Here we show that, in a general setting without star-shaped assumption, obstacle solutions are also viscosity solutions of the local stability condition \eqref{e.stability-condition-intro} and the dynamic slope condition \eqref{e.dynamic-slope-condition-intro}.

To achieve that, we first precisely state that an obstacle solution has the same monotonicity in time as $F$ and is a viscosity solution of a specific obstacle Bernoulli problem at each time.

\begin{lemma}\label{l.OqeOvs}
The obstacle solution \OQE{} for any $(s,t) \subset (0,T) \setminus Z$ satisfies
   \begin{itemize}
\item[(a)]   If $F$ is monotone increasing on $[s,t]$ then $u(t)$ solves in the viscosity sense
\begin{subequations}\label{e.ovs}
\begin{align}\label{e.ovs-a}
\left\{\begin{aligned}
\Delta u(t) &= 0 &&\hbox{in } \Omega(u(t)) \cap U\\
u(t)&>0 &&\hbox{in } \Omega(u(s)) \cap U\\
|\grad u(t)|^2 &\leq 1+\mu_+ &&\hbox{on } (\partial \Omega(u(t)) \cap \partial \Omega(u(s))) \cap U\\
|\grad u(t)|^2 &= 1+\mu_+ &&\hbox{on } (\partial \Omega(u(t))  \setminus \overline{\Omega(u(s))}) \cap U
\end{aligned}\right.
\end{align}
   \item[(b)] If $F$ is monotone decreasing on $[s,t]$ then $u(t)$ solves in the viscosity sense
\begin{align}\label{e.ovs-b}
\left\{\begin{aligned}
\Delta u(t) &= 0 &&\hbox{in } \Omega(u(t)) \cap U\\
u(t)&=0 &&\hbox{in } \overline{U} \setminus \Omega(u(s))\\
|\grad u(t)|^2 &\geq 1-\mu_- &&\hbox{on } (\partial\Omega(u(t)) \cap \partial \Omega(u(s))) \cap U,\\
|\grad u(t)|^2 &= 1-\mu_- &&\hbox{on } (\partial \Omega(u(t))  \cap \Omega(u(s))) \cap U.
\end{aligned}\right.
\end{align}
\end{subequations}
 \end{itemize}

Under the strong star-shapedness assumption of \tref{equivalence}, any function that satisfies the above, $u(t) = F(t)$, and whose initial data $u(0)$ is a solution of \eqref{e.stability-condition-intro}, is the unique obstacle solution with initial data $u(0)$. 
\end{lemma}

\begin{proof}
An obstacle solution satisfies \eref{ovs} due to a standard viscosity solution argument and does not require star-shapedness.

The fact that any function satisfying \eref{ovs} is the unique obstacle solution follows from the uniqueness of the obstacle Bernoulli problem in the strongly star-shaped case: an obstacle solution always exists and by the previous step, satisfies \eref{ovs}, and hence by \tref{bernoulli-obstacle-uniqueness} it is the unique solution satisfying \eref{ovs}.  
\end{proof}

\begin{remark}\label{OVS:MVS} 
 Let us point out that, while most of the analysis in this section regards star-shaped initial data, for general initial data we still have a unique obstacle solution by its definition, and this solution also satisfies \eqref{e.ovs} by the lemma above. However, we do not know whether a solution of \eqref{e.ovs} is also a solution of \eqref{e.stability-condition-intro} and \eqref{e.dynamic-slope-condition-intro}, or vice versa, for general initial data. However formally both \eqref{e.dynamic-slope-condition-intro} and \eqref{e.ovs} come with the same dynamic slope condition at space-time regular points of the free boundary.
\end{remark} 

\begin{remark}\label{r.monotonicity}
Note in particular in case of \lref{OqeOvs}(a) $u$ is monotone increasing on $[s,t]$, and in case of \lref{OqeOvs}(b) $u$ is monotone decreasing in $[s,t]$. Given the condition on $F$, it follows that $u$ also only changes monotonicity finitely many times on any bounded interval. 
\end{remark}

\begin{lemma}\label{l.OvsMvs}
If $u: [0,T] \to C(\overline{U})$ is an obstacle solution \OQE{} on $[0,T]$ that is uniformly non-degenerate at its free boundary, then $u$ solves \eqref{e.stability-condition-intro} and \eqref{e.dynamic-slope-condition-intro} in the sense of \dref{local-stability-visc} and \dref{dynamic-slope-condition-visc}, respectively.
\end{lemma}

\begin{proof}
The state $u(t)$ is harmonic in $\Omega(u(t))$ by definition.  By \rref{immediate-envelope-properties} $u^*$ and $u_*$ are respectively sub and superharmonic in their positivity sets.

Now we aim to check the stability conditions in \dref{local-stability-visc}\partref{stab-part-std}. We first check the lower bound condition $|\nabla u(t)|^2 \geq 1 - \mu_-$ in the viscosity sense on $\partial \Omega(t) \cap U$ by induction on the monotonicity intervals of $F$ provided by assumption \eqref{F-hyp}.  
We know that the condition is satisfied at $t = 0$ by the compatibility of the initial data. Suppose now that we know that $u$ satisfies this condition up to $s \in Z \cup \{0\}$ and consider $t > s$ such that $(s, t) \cap Z = \emptyset$.
\begin{itemize}
\item 
If $F$ is decreasing on $[s, t]$ then the lower bound on $|\nabla u(t)|^2$ follows from \eref{ovs-b} immediately. 
\item 
Otherwise $F$ is increasing on $[s, t]$. In this case $u(t)$ are minimal supersolutions above $u(s)$ for $t_0 \in (s,t]$ and so, by \lref{lip-bdry-nondegen}, they are uniformly non-degenerate. On $(\partial \Omega(t) \setminus \overline{\Omega(s)}) \cap U$ we get the same bound from \eref{ovs-a} since $|\nabla u(t)|^2 = 1 + \mu_+ \geq 1- \mu_-$. 
Finally, on $(\partial \Omega(t) \cap \partial \Omega(s)) \cap U$, as $u(t) \geq u(s)$, any test function touching $u(t)$ from above in $\overline{\{u(t) > 0\}}$ at $x \in (\partial \Omega(t) \cap \partial \Omega(s))\cap U$ also touches $u(s)$ and hence $D_+ u(t, x) \subset D_+ u(s, x)$. Hence $u(t)$ is a solution $|\nabla u(t)|^2 \geq 1 - \mu_-$ at $x$ in the viscosity sense since $u(s)$ is by the induction hypothesis. 
\end{itemize}
By induction we can therefore extend the property to all $t \in [0, T]$ since $Z$ is discrete.
An analogous argument gives $|\nabla u(t)|^2 \leq 1 + \mu_+$.

By the assumption of uniform non-degeneracy, \rref{envelopes-stability} yields \dref{local-stability-visc}\partref{local-stab-sub-part}. We note that the assumption is actually needed only when $F$ is decreasing. When $F$ is increasing, $u(t)$ is a minimal supersolution above $u(s)$ and hence non-degenerate by \lref{lip-bdry-nondegen}.

We now check the remaining dynamic slope conditions. First consider monotone increasing interval $[s,t]$ of $F$ and we aim to check the advancing condition \dref{dynamic-slope-condition-visc}\partref{advancing-cond}. Suppose the positive cone condition \eqref{cone_positive} holds for $x_0 \in \partial \Omega(u^*(t_0)) \cap U$ for some $s < t_0 \leq t$. By the positive cone condition and monotonicity there is a neighborhood $B_{r}(x_0) \subset  \overline{\Omega(u(s))}^\complement$, and thus we have $|\nabla u|^2 \geq 1+\mu_+$ on $\partial \Omega(u)$ in $(s,t] \times B_r(x_0)$ by \eref{ovs-a}.  If $t_0 = t$ then $u^*(t) = u(t)$ because $u(t)$ is a local maximum by increasing monotonicity to the left and decreasing monotonicity to the right, so we can reduce to $t_0 \in (s,t)$. 
Then, by uniform non-degeneracy of minimal supersolutions \lref{lip-bdry-nondegen} and \rref{envelopes-stability}, we can conclude that $|\grad u^*|^2 \geq 1+\mu_+$ on  $\partial \Omega(u^*)$ in $(s,t) \times B_r(x_0)$, in particular it also holds at $(t_0,x_0)$.

Next we check that the receding condition \dref{dynamic-slope-condition-visc}\partref{receding-cond} holds for monotone decreasing interval $[s,t]$ of $F$. The argument is mostly the same with a subtle difference which we need to point out. Suppose the negative cone condition \eref{cone_negative} holds for $x_0 \in \partial \Omega(u_*(t_0)) \cap U$ for some $s < t_0 \leq t$. As before there is a neighborhood $B_{r}(x_0) \subset  \Omega(u(s))$, and thus we have $|\nabla u|^2 \leq 1-\mu_-$ on $\partial \Omega(u)$ in $(s,t] \times B_r(x_0)$ by \eref{ovs-b}.  If $t_0 = t$ then $u_*(t) = u(t)$ so we reduce to the case $t_0 \in (s,t)$. Then by \rref{envelopes-stability-supersoln}, we can conclude that $|\grad u_*| \leq 1 - \mu_-$ on  $\partial \Omega(u^*)$ in $(s,t] \times B_r(x_0)$, in particular it also holds at $(t_0,x_0)$.  Notice that we do \emph{not} need uniform non-degeneracy, which is not known in general for maximal subsolutions, because we are testing a supersolution condition.
\end{proof}

\subsection{Spatial regularity} 
We start with star-shapedness of obstacle solutions.
\begin{lemma}
\label{l.star-shapedness}
Suppose that $u$ is an obstacle solution \OQE{}, $K = \R^d \setminus U$ and $\Omega_0 = \{u_0>0\}$ are strongly star-shaped, and $F$ satisfies \eqref{F-hyp}. Then $u(t)$ is strongly star-shaped for every $t \geq 0$ with respect to the same set of centers as $K$ and $\Omega_0$.
\end{lemma}

\begin{proof}
We show the star-shapedness only with respect to the origin assuming that $K$ and $\Omega_0$ are star-shaped with respect to the origin. The proof with respect to other centers can be done analogously by translation.

We will use induction on the monotonicity intervals of $F$. Suppose that $t > 0$ with $(0, t) \cap Z \neq \emptyset$. 
Consider the case $F(t) > F(0)$.
First note that for $\lambda < 1$
\[ v(x) =
\begin{cases}
\min\big(u(t, \lambda x), F(0)\big) & x \in \lambda^{-1} U,\\
F(0) & x \in U \setminus \lambda^{-1} U,
\end{cases}
\]
is a supersolution of \eqref{e.stability-condition-intro}.  Since $\{u(0)>0\}$ is star-shaped and $\{u(t)>0\} \supset \{u(0)>0\}$ also
\[ \lambda^{-1}\{u(t) >0\} \supset \{u(0) >0\}.\]
Thus $v$ is a supersolution of \eqref{e.stability-condition-intro} above $u(0)$ and so
\[ u(t, \lambda \cdot) \geq u(t) \qquad \text{in }\lambda^{-1} U\]
by the minimality of $u(t)$ in the definition of \OQE.
Since $\lambda>0$ was arbitrary, $u(t)$ is star-shaped (as are all of its super-level sets) with respect to $0$.

In the case $F(t) < F(0)$ we consider the dilation with $\lambda > 1$ and easily check that $u(t, \lambda \cdot) \leq u(0)$ in $U$ and $u(t, \lambda \cdot)$ is a subsolution of \eqref{e.stability-condition-intro} and hence $u(t, \lambda \cdot) \leq u(t)$ in $U$ by the maximality of $u(t)$ in the definition of \OQE. In particular, $u(t)$ is star-shaped with respect to $0$.

By translation we have star-shapedness at any point at which $K$ and $\{u_0 > 0\}$ are star-shaped.
Finally by induction on the monotonicity intervals of $F$, we can continue the star-shapedness property up to any $t > 0$.
\end{proof}

An obstacle solution is a solution of the Bernoulli obstacle problem on each monotonicity interval of $F$, where the obstacle is the solution at the beginning of the interval. This allows us to use the regularity of the obstacle problem, \sref{bernoulli-reg}, only a finite number of times to establish the regularity of obstacle solutions at an arbitrary time.

\begin{lemma}
\label{l.space-regularity}
Suppose that $u$ is an obstacle solution \OQE\ with $F$ satisfying \eqref{F-hyp}, $K = \R^d \setminus U$ has $C^2$ boundary, $\Omega_0$ has a $C^{1,\alpha}$ boundary, and both $K$ and $\Omega_0$ are strongly star-shaped. Then for any $0 < \beta < \min\{\alpha,\frac{1}{2}\}$ the domains $\{u(t) > 0\}$ are $C^{1,\beta}$ and $u(t) \in C^{1, \beta}(\overline{\{u(t) > 0\} \cap U})$ uniformly in $t \in [0, T]$. 
\end{lemma}

\begin{proof}

Due to the regularity of $K = \R^d \setminus U$ the function $u(t)$ is in $C^2(\{u(t) > 0\} \cup \partial U)$. We need to show the regularity at the free boundary of $u(t)$.

By \lref{star-shapedness} the sets $\{u(t) > 0\}$ are uniformly strongly star-shaped, and therefore they are uniformly Lipschitz in $[0, T]$. Let $t > 0$ be such that $(0, t) \cap Z = \emptyset$. Therefore by the regularity of the Bernoulli obstacle problem \tref{bernoulli-reg} with $C^{1, \alpha}$ obstacle $\{u_0 > 0\}$ and any $ 0 < \beta_0 < \min\{\alpha,\frac{1}{2}\}$ the domain
$\{u(t) > 0\}$ is $C^{1,\beta_0}$ and $u(t) \in C^{1, \beta_0}(\overline{\{u(t) > 0\} \cap U})$.

Consider now $t > 0$ such that $(0, t) \cap Z = \{s_1\}$. By the previous step $\{u({s_1}) > 0\} \in C^{1, \beta_0}$. Applying the regularity of the Bernoulli obstacle problem with $\alpha = \beta_0$ for any $0 < \beta_1 < \beta_0$ we find that $u(t)$ has regularity $C^{1, \beta_1}$.

We can continue inductively on the monotonicity intervals of $F$ up to $t = T$, just taking a strictly decreasing sequence of $\min\{\alpha,\frac{1}{2}\} > \beta_0>\beta_1> \ldots > \beta_{\# (0, T) \cap Z} = \beta$.
\end{proof}

\subsection{Time regularity} The regularity in space \lref{space-regularity} of a star-shaped obstacle solution yields in particular that $u(t)$ is Lipschitz locally uniformly in time. This in particular implies (with the nondegeneracy of the solutions of the Bernoulli problem) that its support moves in a Lipschitz way as well.

\begin{lemma}
\label{l.time-regularity}
Suppose that $u$ is an obstacle solution \OQE{}, $K = \R^d \setminus U$ has $C^2$ boundary, and $K$ and $\Omega_0 = \{u_0>0\}$ are strongly star-shaped and bounded, and $u_0$ satisfies \eqref{e.stability-condition-intro}. Then
\begin{align}
\label{lip-in-time}
|u(t, x) - u(s, x)| \leq C \|F'\|_\infty |t - s| \ \hbox{ in } \ U
\end{align}
where $C = C(L, R, \min F)$ with $L$ the Lipschitz constant of $u$ in space and $R := \max_{x \in K \cup \Omega_0} |x|$.
\end{lemma}

\begin{proof}
\textit{Step 1.} By the minimality of $u(t)$, including $u(0)$ by \eqref{e.stability-condition-intro},
\begin{align*}
u(t, x) \leq \Big(F(t) + (1 + \mu_+) (R - |x|)\Big)_+, \qquad t \in [0, T], x \in \overline U,
\end{align*}
since the right-hand side is a supersolution.
Therefore
\begin{align}
\label{Omegat-radius}
x \in \Omega(t) \quad \Rightarrow \quad |x| \leq \frac{F(t)}{1 + \mu_+} + R.
\end{align}

\textit{Step 2.}
Fix $(s, t) \subset \R \setminus Z$ with $F(s) < F(t)$. Let $\lambda := F(s)/F(t) < 1$. Define
\begin{align*}
v(x) := 
\begin{cases}
\lambda^{-1} u(s, \lambda x),  &x \in \lambda^{-1} \overline U,\\
F(t), &\text{otherwise}.
\end{cases}
\end{align*}
By the choice of $\lambda$, $v$ is continuous and hence a supersolution on $U$ with boundary data $F(t)$.
By minimality, $u(t) \leq v$.

Fix $\lambda x \in \Omega(s) \cap U$.
\eqref{Omegat-radius} implies $|x| \leq \lambda^{-1} (\frac{F(s)}{1 + \mu_+} + R)$.
Recall also that $u(s, \lambda x) \leq F(s)$. 
We have
\begin{align*}
v(x) &= \lambda^{-1} u(s, \lambda x) = u(s, \lambda x) + (\lambda^{-1} - 1) u(s, \lambda x)\\
&\leq u(s, x) + \|\nabla u(s)\|_\infty (1 -\lambda) |x| + (\lambda^{-1} - 1)F(s)\\
&\leq u(s, x) +  (\lambda^{-1} - 1) \left(L (\frac{F(s)}{1 + \mu_+} + R) + F(s)\right).
\end{align*}
Since $u_{t} \leq v$ we conclude that
\begin{align*}
0 \leq u(t, x) - u(s, x) \leq \left(\frac{L}{1+ \mu_+} + 1 + \frac{R}{F(s)}\right) (F(t) - F(s)).
\end{align*}

For $x \in U \setminus \lambda^{-1} U$ we note that $\dist(x, \partial U) \leq (\lambda^{-1} - 1) R = \frac{R}{F(s)} (F(t) - F(s))$ and therefore
\begin{align*}
0 \leq u(t, x) - u(s, x) &\leq F(t) - F(s) + \|\nabla u(s)\|_\infty \dist(x, \partial U)\\
&\leq \left(1 + \frac{LR}{F(s)}\right) (F(t) - F(s)).
\end{align*}

\textit{Step 3.}
If $(s, t) \subset \R \setminus Z$ with $F(s) > F(t)$, we have $\lambda = F(s) / F(t) > 1$ and $v \leq F(t)$. Hence $v$ is a subsolution and therefore by maximality $u(t) \geq v$. We only need to consider $x \in \Omega(s) \cap U$. This time
\begin{align*}
v(x) &= \lambda^{-1} u(s, \lambda x) = u(s, \lambda x) + (\lambda^{-1} - 1) u(s, \lambda x)\\
&\geq u(s,x) - \|\nabla u(s)\|_\infty (\lambda - 1) |x| - (1 - \lambda^{-1})F(s)\\
&\leq u(s,x) -  \left(L (\frac{F(s)}{1 + \mu_+} + R\right)\frac{F(s) - F(t)}{F(t)}  - F(t) + F(s).
\end{align*}
Hence by $u(t) \geq v$
\begin{align*}
0 \geq u(t,x) - u(s,x) \geq L \left(\frac{F(s)}{(1 + \mu_+)F(t)} + 1 + \frac{R}{F(t)}\right)(F(s) - F(t)).
\end{align*}

For general $(s, t)$, we recover the estimate \eqref{lip-in-time} by triangle inequality.
\end{proof}

Using the nondegeneracy of cone monotone solutions of the Bernoulli problem, \lref{lip-bdry-nondegen}, we can deduce Lipschitz regularity of $\Omega(u(t))$ in time.

\begin{corollary}
\label{c.omega-time-regularity}
Under the assumptions of \lref{time-regularity}, the sets $\Omega(u(t))$ are locally Lipschitz in time with respect to the Hausdorff distance.
\end{corollary}

\begin{proof}
Let us fix $s \neq t$. Let $z \in \partial \Omega(u(t))$ be the point that maximizes the distance from $\partial \Omega(u(s))$. Let us denote $r := \dist(z, \partial \Omega(u(s)))$. 

Suppose that $z \notin \Omega(u(s))$. We have $B_r(z) \cap \Omega(u(s)) = \emptyset$ and hence $u_s = 0$ on $B_r(z)$. On the other hand, the free boundary is locally a Lipschitz graph and so the non-degeneracy (note that $u(t)$ satisfies \eqref{e.stability-condition-intro} by \lref{OvsMvs}) for the Bernoulli problem in this setting \lref{lip-bdry-nondegen} yields $\sup_{B_r(z)} u(t) \geq c_0 r$. By the Lipschitz continuity of $u$ in time, we deduce
\begin{align*}
c_0 r \leq\sup_{B_r(z)} u(t) \leq \tilde L |t - s|,
\end{align*}
where $\tilde L$ is the Lipschitz constant in \eqref{lip-in-time}. In particular, $r \leq \frac{\tilde L}{c_0} |t - s|$.

If $z \in \Omega(u(s))$ then $B_r(z) \subset \Omega(u(s))$. Using the non-degeneracy, Harnack inequality and the Lipschitz regularity, we have
\begin{align*}
C_H c_0 \frac r2 \leq C_H \sup_{B_{r/2}(z)} u(s) \leq u(s,z) \leq \tilde L |t - s|,
\end{align*}
which yields $r \leq \frac{2 \tilde L}{C_H c_0} |t - s|$.
\end{proof}

Finally we mention that the time regularity and space regularity can be interpolated to get time regularity of $\grad u$.

\begin{corollary}\label{c.derivative-time-cont}
Under the assumptions of \lref{space-regularity} for any $0 < \beta < \min\{\alpha,\frac{1}{2}\}$
 \[u\in C^{0,1}_tC_x \cap L^\infty_t C^{1,\beta}_x \ \hbox{ so by interpolation } \ u \in C^{\frac{\beta}{1+\beta}}_tC_x^1.\]
In particular, if $\alpha \geq \frac{1}{2}$ then $\grad u \in C^{\frac{2}{3}-}_tC^0_x$.
\end{corollary}
\begin{proof}
  Since $\grad u(t)$ are in $C^\beta(\overline{\Omega(u(t))})$ uniformly in $t$ we can use the uniform convergence of the difference quotients and the Lipschitz continuity in time of $u(t)$ from \lref{time-regularity}
\begin{align*} 
(\grad u(t,x) - \grad u(s,x)) \cdot e &= \frac{1}{h}(u(t,x+he) - u(t,x)) + \frac{1}{h}(u(s,x+he) - u(s,x)) +O(h^\beta) \\
&= O(h^{-1}|t-s|)+O(h^\beta).
\end{align*}
By choosing $h = |t-s|^{\frac{1}{1+\beta}}$ to match the orders of the two terms, we recover the desired gradient continuity.
\end{proof}

\subsection{Comparison principle in star-shaped setting}
Now we state a comparison principle that closes our discussion of the characterization of obstacle solutions using the local stability and dynamic slope conditions. We will utilize this comparison principle later on in \sref{MM-visc-prop}, when we discuss a specific energy solution generated by minimizing movements. 

\begin{proposition}\label{p.MVS-OVS-comparison}
Suppose that $w$ and $u$  are respectively a viscosity subsolution and a LSC viscosity supersolution (or a supersolution and an USC subsolution) of \eqref{e.stability-condition-intro} and \eqref{e.dynamic-slope-condition-intro} in $[0,T]$ with strongly star-shaped initial data, with the ordering $w(0) \leq u(0)$ (or $w(0) \geq u(0)$). Further suppose that $w$  satisfies the following: 
\begin{itemize}
\item[(a)]$\{w(t)>0\}$ is bounded and strongly star-shaped at each time, its boundary is uniformly $C^{1,\alpha}$ in space and Lipschitz in time, and $w(t)$ is $C^{1}$ in $\overline{\Omega(w(t))}$, and $\grad w(t)$ is continuous in time on $\bigcup_{t\in[0,T]} \overline{\Omega(w(t))}$.
\item[(b)] There exists a finite partition $\{t_i\}_{i=0}^{i=n}$ of $[0,T]$ such that $w$ is monotone in $[t_i, t_{i+1}]$.
\end{itemize} 
Then $w\leq u$ (or $w \geq u$).
\end{proposition}

Let us mention that an obstacle solution with star-shaped initial data satisfies (a) in \pref{MVS-OVS-comparison} due to \lref{space-regularity} and \lref{time-regularity} and  (b) due to \rref{monotonicity}. It is also a viscosity solution of \eqref{e.stability-condition-intro} and \eqref{e.dynamic-slope-condition-intro} by \lref{OvsMvs}. The uniform non-degeneracy follows from the boundary regularity and \lref{lip-bdry-nondegen}(i).  Thus, from the comparison principle \pref{MVS-OVS-comparison}, we can derive the uniqueness:

\begin{corollary}\label{c.MvsOvs}
 Any viscosity solution of \eqref{e.stability-condition-intro} and \eqref{e.dynamic-slope-condition-intro} with a strongly star-shaped initial data $u(0)$ that is continuous at $t = 0$ in the sense that $u_*(0) = u(0) = u^*(0)$, and boundary data satisfying \eqref{F-hyp} is the unique star-shaped obstacle solution with the same initial data and boundary data.
 \end{corollary}

Let us now turn to the proof of \pref{MVS-OVS-comparison}.
As is typical with viscosity solutions, one of the main difficulties in the comparison principle of weak solutions is the lack of regularity.  We follow a typical approach based on sup / inf convolutions in order to regularize at the point where the first touching occurs.  The novel difficulty has to do with the weak notion of positive velocity and the difficulties stemming from crossings which occur with zero velocity.  We can use sup / inf convolutions with decreasing radius to accelerate the free boundary velocity, but a delicate argument relying on the regularity of one of the solutions is still needed to actually guarantee that the first touching point has a positive velocity.

\begin{proof}[Proof of \pref{MVS-OVS-comparison}]
Let us assume that $w$ is a viscosity subsolution of \eqref{e.stability-condition-intro} and \eqref{e.dynamic-slope-condition-intro} with regularity properties in assumptions (a) and (b), and $u$ is a LSC supersolution. We will show that 
\[
w(t,\lambda \cdot) \leq u(t,\cdot) \hbox{ for } t>0 \hbox{ and } \lambda>1,
\]
which yields $w \leq u$ in the limit $\lambda \searrow 1$. The proof for the other half of the statement is parallel, arguing the opposite inequality with $\lambda < 1$.

Fix $\lambda > 1$. Define
\[
W(t,x) := \sup_{y \in \overline B_{\rho(t)}(x)} w(t,\lambda y).
\]
Here $\rho(t) = \rho_0 e^{-t}$, and
$\rho_0$ is chosen proportional to $(\lambda-1)$ so that $W(0) < u(0)$ in $\overline{\{W(0) > 0\}}$, which is possible by strong star-shapedness. The main motivation for the sup-convolution is the effective reduction of the normal velocity of $W$ by $\rho'(t) < 0$ compared to $w$. Hence if $W$ and $u$ touch each other, either the normal velocity of $u$ is negative or $w$ has positive normal velocity near the contact, and therefore the dynamic slope condition can be used in either situation to arrive at a contradiction. Recall that $W$ is subharmonic in its positive set.

We would like to show that $W(t) < u(t)$ in $\overline{\{W(t) > 0\}}$ for all $t \geq 0$. 
Define the first contact time
\begin{align*}
t_0 = \inf\big\{t > 0: W(t) \not< u(t) \text{ in $\overline{\{W(t) > 0\}}$}\big\}.
\end{align*}
Clearly $t_0 > 0$ by lower semi-continuity of $u$ in time.
If $t_0$ is finite, we want to get a contradiction. The only non-standard part here lies in the fact that the support of $u$ may jump in time, so let us discuss this point carefully. Since we assume that $W$ is star-shaped, $W$ and its support continuously decrease as $\lambda$ increases. The support of $\{W(t_0)>0\}$ indeed can be made arbitrarily close to $\overline B_{\rho(t_0)}(0)$, in particular a subset of the fixed boundary $U^\complement$ for $u$, if $\lambda$ is sufficiently large.  Thus if $\{W(\cdot,t_0)>0\} \not\subset\{u(\cdot, t_0)>0\}$ due to discontinuity in time in $u$ at $t_0$, we increase $\lambda$ in the definition of $W$ until $\{W(t_0)>0\}$ precisely touches $\{u(t_0)>0\}$ from inside, so that there exists a point $x_0 \in \partial\{W(t_0)>0\} \cap \partial\{u(t_0)>0\}$.  We point out that $\rho_0$ does not need to be changed even if we need to increase $\lambda$.  Note that $W\leq u$ for $t\leq t_0$ for this larger choice of $\lambda$ due to the definition of $t_0$ and $W(t_0) - u(t_0)$ being subharmonic in $\{W(t_0) > 0\}$. In particular $u(t_0) - W(t_0)$ has a minimum at $x_0$.

By the definition of $W$ there is a point $y_0 \in \partial \{w(t_0, \lambda \cdot) > 0\}$ with $y_0 \in \overline{B_{\rho(t_0)}}(x_0)$, and
\begin{align*}
W(t_0, x) \geq w(t_0, \lambda (x - x_0 + y_0)) =: \phi(x).
\end{align*}
As $u(t_0) - \phi$ has a minimum at $x_0$, we deduce that 
\begin{align}
\label{gradw-inD-}
\lambda \nabla w(t_0, \lambda y_0)  = \nabla \phi(x_0) \in D_- u(x_0)
\end{align}
and so
\begin{equation}\label{e.gradw-inD-2}
    \lambda^2 |\nabla w(t_0, \lambda y_0)|^2 \leq 1 + \mu_+.
\end{equation}

By our assumption (b), we know that $w$ is either monotone increasing or decreasing in a closed time interval $[t_{-1}, t_0]$ with $t_{-1}<t_0$. If $w$ is monotone decreasing in $[t_{-1},t_0]$, then $W$ has strictly decreasing support in $[t_{-1},t_0]$, namely for $c = -\rho'(t_0)>0$ we have
\begin{equation*}
B_{c(t_0-t)}(x_0) \subset\{W(t)>0\} \subset\{u(t)>0\} \hbox{ for } t_{-1}\leq t<t_0.
\end{equation*}
In particular the negative cone condition \eref{cone_negative} is satisfied for $u$ at $(t_0, x_0)$, and so as $u$ is a supersolution of \eqref{e.dynamic-slope-condition-intro} we conclude by \eqref{gradw-inD-} that $\lambda^2 |\nabla w|^2 (t_0, \lambda y_0) \leq 1 - \mu_-$. However, as $w$ is itself a subsolution of \eqref{e.stability-condition-intro}, we have $\lambda^2 |\nabla w|^2(t_0, \lambda y_0) \geq \lambda^2(1 - \mu_-)$, which is a contradiction.

Now suppose that $w$ is monotone increasing in $[t_{-1},t_0]$. It is possible that in this case $u$ still satisfies the negative cone condition \eref{cone_negative} for some $c<0$ and $r_0>0$, in which case one can still proceed as in the above case. So now suppose that \eref{cone_negative} is not true for $c := -\frac{1}{3}\rho'(t_0) > 0$ and any $r_0>0$. This means that along an increasing time sequence $t_n$ converging to  $t_0$, there is $x_n$ such that 
\[
x_n \in \{u(t_n)=0\}\cap B_{c(t_0-t_n)}(x_0).
\]

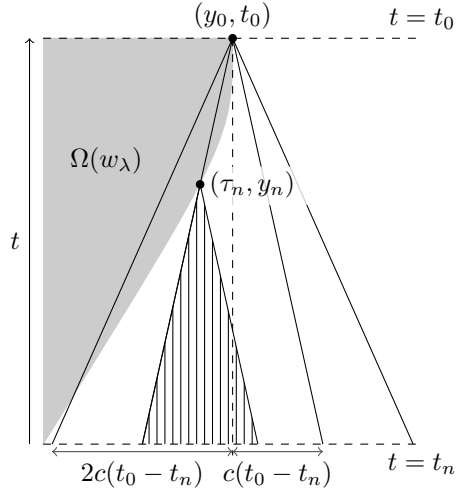
\begin{figure}
\begin{tikzpicture}[scale = 2]

\def\yscale{2.7}
\def\xscale{.6}

\filldraw[pattern = vertical lines] (-1*\xscale,-1*\yscale) -- (-.36*\xscale,-.36*\yscale) -- ({-.36*\xscale+\xscale/\yscale*(\yscale*1 - \yscale *.36)},-1*\yscale);

\filldraw[color = gray!40] (-2.1*\xscale,0) -- (0,0) .. controls (0,-.25*\yscale) .. (-2.1*\xscale,-1*\yscale);

\filldraw[black] (0,0) circle (.025);
\draw[black] (-1*\xscale,-1*\yscale) -- (0,0) -- (1*\xscale,-1*\yscale);
\draw[black] (-2*\xscale,-1*\yscale) -- (0,0) -- (2*\xscale,-1*\yscale);

\draw[black,dashed] (-2.1*\xscale,0) -- (2.1*\xscale,0);

\draw[->] (-2.25*\xscale,-1*\yscale) -- (-2.25*\xscale,0);

\filldraw (-.36*\xscale,-.36*\yscale) circle (.025);

\node[left] at (-2.25*\xscale,-.5*\yscale) {$t$};
\node at (-1.4*\xscale,-.3*\yscale) {$\Omega(w_\lambda)$};
\node[above] at (0,0) {$(y_0,t_0)$};

\draw[dashed] (0,0)--(0, -1*\yscale-0.05-0.025);

\draw[<->,black!70!white] (0,-1*\yscale-.05) -- (1*\xscale,-1*\yscale-.05);

\node[below] at (.5*\xscale,-1*\yscale - .05) {$c(t_0-t_n)$};


\draw[<->,black!70!white] (-2*\xscale,-1*\yscale-.05) -- (0,-1*\yscale-.05);

\node[below] at (-1*\xscale,-1*\yscale - .05) {$2c(t_0-t_n)$};

\draw[black,dashed] (-2.1*\xscale,-1*\yscale) -- (2.1*\xscale,-1*\yscale);
\node[below] at (2.1*\xscale,-1*\yscale) {$t=t_n$};
\node[above] at (2.1*\xscale,0) {$t=t_0$};

\filldraw[white,opacity = .8] ({(-.36+.1)*\xscale},{(-.36-.05)*\yscale}) rectangle ++(.7,.25);
\node[right] at (-.36*\xscale,-.36*\yscale) {$(\tau_n,y_n)$};

\end{tikzpicture}
\caption{Space-time picture displaying the definition of $(\tau_n,y_n)$.}
\label{f.comparison-pics}
\end{figure}

Now we consider what this means in terms of $w$. Recall $y_0$ above and consider $y$ such that $|y-y_0| \leq 2c(t_0-t_n)$. Then 
\begin{align*}
|y- x_n| &\leq |y-y_0| + |y_0 - x_0| + |x_0 - x_n| \\
&\leq  2c(t_0-t_n) + \rho(t_0) + c(t_0-t_n) \\
&= \rho(t_0) - \rho'(t_0) (t_0 - t_n)\\
&\leq \rho(t_n),
\end{align*}
where the last inequality is by convexity of $\rho(t)$. Thus, by definition of $W$, $w(t_n, \lambda y) \leq W(t_n, x_n) \leq u(t_n, x_n) = 0$. Namely $\overline B_{2c(t_0-t_n)} (y_0) \subset \{w(t_n, \lambda \cdot) = 0\}$.  


Now let 
\[
\tau_n:= \sup\{ s\in [t_n,t_0): \overline{B}_{c(t_0-s)}(y_0)\subset \{w(s, \lambda \cdot)=0\} \}.
\]
See \fref{comparison-pics}.  Since $\{w(t)=0\}$ evolves continuously in time by assumption (a),  we have $\tau_n >t_n$, and if $\tau_n<t_0$ then there is a contact point $y_n\in \partial\{w(\tau_n)>0\} \cap \partial B_{c(t_0-\tau_n)}(y_0)$. If the order is preserved up to $t=t_0$, i.e. $\tau_n = t_0$,  then we set $y_n=y_0$. From the definition of $(y_n, \tau_n)$ we see that positive cone condition \eqref{cone_positive} holds for $w(\cdot, \lambda \cdot)$ at $(\tau_n, y_n)$ and therefore for $w$ at $(t_n, \lambda y_n)$ with cone speed $\lambda c$. Since $w$ is a subsolution of \eqref{e.dynamic-slope-condition-intro} we have $|\nabla w|^2 (t_n, \lambda y_n) \geq 1 + \mu_+$.

Since $(\tau_n, y_n) \to (t_0, y_0)$, by the gradient continuity in assumption (a), we deduce
\[\lambda^2|\nabla w|^2(t_0, \lambda y_0) = \lim_{n\to\infty} \lambda^2|\nabla w|^2(\tau_n, \lambda y_n) \geq \lambda^2(1+\mu_+).
\] 
But this contradicts \eref{gradw-inD-2}.
Hence we conclude. 
\end{proof}

\section{Comparison properties of the minimizing movements scheme}\label{s.MM-visc-prop} 

In this section we discuss the comparison properties of the minimizing movement scheme \eqref{e.minimizing-movement-scheme-solution} and its $\delta \to 0$ limit.
In particular, we will show that the approximate solutions generated by this scheme will also satisfy the dynamic slope condition \eqref{e.dynamic-slope-condition-intro} in the viscosity sense.  Since the local stability condition \eqref{e.stability-condition-intro} is shown in \cite{FKPi}, in the star-shaped case we will be able to apply the comparison principle \pref{MVS-OVS-comparison} proved in \sref{comparison} to show convergence to the unique obstacle solution \OQE, see \tref{oqe-MM-relation}.  As a corollary we obtain that the obstacle solution \OQE{} is an energy solution.

 These results will complete the proof of our main theorem in the introduction, \tref{main-1}.
 
 \begin{remark}\label{remark:regularity}
The very recent result of Ferreri and Velichkov \cite{FerreriVelichkov} implies that $C^{1,\beta}$ regularity is propagated in the discrete scheme \eref{minimizing-movement-scheme}. However, at least with a naive application of their result, the constants would blow up as $\delta \to 0$ since the regularity theorem needs to be applied $O(\frac{1}{\delta})$ times in order to get to a nontrivial positive $t$ in the limit. In the star-shaped case we go around this difficulty due to the equivalence with the obstacle solution \OQE{} evolution. This equivalence allows us to apply the regularity theorem only finitely many times, at each monotonicity change, as in \lref{space-regularity}.
 \end{remark}

\begin{lemma}\label{l.minmov-MVS}
The approximate solution $u_\delta$ generated by the scheme \eqref{e.minimizing-movement-scheme} are viscosity solutions of the local stability \eqref{e.stability-condition-intro} and dynamic slope \eqref{e.dynamic-slope-condition-intro} conditions on $[0,T]$ in the sense of \dref{local-stability-visc} and \dref{dynamic-slope-condition-visc}.
\end{lemma}
\begin{proof}
The local stability condition \eqref{e.stability-condition-intro} in the sense of \dref{local-stability-visc} sense follows from \cite[Lemma~3.2
, Lemma~3.3
]{FKPi}. In particular, \dref{local-stability-visc}\partref{local-stab-sub-part} follows from \rref{envelopes-stability} and the uniform nondegeneracy of energy minimizers.

For \eqref{e.dynamic-slope-condition-intro}, we only check the subsolution condition since the supersolution case is symmetrical.

Suppose that $\Omega(u_\delta^{*}(t))$ has positive normal velocity $V_n(t_0,x_0)>0$ at some point $x_0 \in \partial \Omega(u_\delta^{*}(t_0))$ in the sense of \dref{positive-velocity-cone}, i.e.
\[\{x: \ |x-x_0| \leq c (t-t_0)\} \subset \Omega(u_\delta^{*}(t))^{\complement} \ \hbox{ for } \ t_0 - r_0 \leq t < t_0.\]
Since $u_\delta$ is constant on open intervals of the form $(k\delta,(k+1)\delta)$ this implies that $t_0 = k\delta$ for an integer $k$, that $u_\delta^*(t_0) = u^k_\delta$, and that $u^*_\delta(t) =  u^{k-1}_\delta$ for $t_0 - \delta \leq t < t_0$.  Since $u^k_\delta$ is a minimizer of
\[\mathcal{E}(u^{k-1}_\delta,v) = \mathcal{J}(v) + \Diss(u^{k-1}_\delta,v) \ \hbox{ over } \ v \in F(k\delta) + H^1_0(U)\]
and we have just established that $x_0 \in \R^d \setminus \overline{\Omega(u^{k-1}_\delta)}$ we can apply \cite[Lemma~3.2
, Lemma~3.3
]{FKPi} to conclude that
\[|\grad u(t_0,x_0)|^2 \geq 1+ \mu_+ 
\]
in the viscosity sense.
\end{proof}

\begin{remark}
The $u_\delta(t)$ are also viscosity solutions of \eqref{e.dynamic-slope-condition-intro} in the ``standard" sense of touching from above and below by $C^1$ test functions in space-time. This notion, unlike the notion in \dref{dynamic-slope-condition-visc} which defines positive velocity based on cones, behaves well with respect to the limit $\delta \to 0$.  One can apply the standard method of upper and lower half-relaxed limits, of Barles and Perthame \cite{BarlesPerthame}, to show that $\overline{u}^*$ and $\underline{u}_*$ are respectively sub and supersolutions of \eqref{e.dynamic-slope-condition-intro} in the standard viscosity sense.  

We should remark that the upper and lower-half relaxed limits $\overline{u}^*$ and $\underline{u}_*$ are not necessarily the same as the upper and lower semicontinuous envelopes $u^*(t)$ and $u_*(t)$ of the pointwise limit $u(t)$ of an energy solution constructed in 
\cite[Th.~1.3
]{FKPi}. It would be interesting to know whether the pointwise in time (subsequential) limits $u(t)$ are also viscosity solutions of \eqref{e.dynamic-slope-condition-intro} in general either in the standard sense or in the cone sense of \dref{dynamic-slope-condition-visc}. Below in \tref{oqe-MM-relation} we will show that it is so in the star-shaped case.

\end{remark}

Now we can apply the previous viscosity solution property in combination with the comparison theorem in \sref{comparison} to show that, in the star-shaped case, the limit of the minimizing movements scheme is the same as the unique obstacle solution.  The convergence is in fact quantitative.
\begin{theorem}\label{t.oqe-MM-relation}
Suppose that $\Omega_0$ is a $C^{1,\alpha}$ strongly star-shaped region, $K = \R^d \setminus U$ is compact and strongly star-shaped, $F$ is positive Lipschitz continuous, and $F$ changes monotonicity at most finitely many times on $[0,T]$.  Call $v$ the obstacle solution \OQE{} and $u_\delta$ to be a solution of the time incremental scheme \eref{minimizing-movement-scheme} on $[0,T] \times U$. Then for any $t \in [0, T]$
\[\|u_\delta(t) - v(t)\|_{L^\infty(U)} \leq C(\mu_\pm,d)(\exp(\|(\log F)'\|_\infty \delta) - 1)\|F\|_\infty \to 0 \ \hbox{ as } \ \delta \to 0.\] 
Thus the obstacle solution \OQE{} is the unique minimizing movements solution of the energetic evolution (\cite[Def.~1.2
]{FKPi}), in particular any minimizing movements solution is also a viscosity solution of \eqref{e.stability-condition-intro} and \eqref{e.dynamic-slope-condition-intro}.
\end{theorem}

\begin{proof}
We prove the convergence by comparison principle \pref{MVS-OVS-comparison} with $w = v$ the obstacle solution which, by \cref{derivative-time-cont}, is in $C^{0,1}_tL^\infty_x \cap L^\infty_tC^{1,\beta}_x$ on $[0,T] \times U$.

Let $M(\delta) = \sup_t\sup_{h \in [0,\delta]} F(t-h)/F(t)$, note that $M(\delta) \leq \exp(\|(\log F)'\|_\infty \delta)$.  Then
\[\varphi_\delta(t,x) := M(\delta)v(t,M(\delta)^{-1} x)\]
is a $C^{0,1}_tL^\infty_x \cap L^\infty_tC^{1,\beta}_x$ supersolution of \eqref{e.stability-condition-intro} and \eqref{e.dynamic-slope-condition-intro} on $[0,T]$ with star-shaped level sets and so the comparison principle \pref{MVS-OVS-comparison} implies that $u_\delta \leq \varphi_\delta$ on $[0,h] \times U$.  Note that we have chosen $M(\delta) \geq 1$ so that for $x \in K$ also $M(\delta)^{-1} x \in K$ since $K$ star-shaped and
\[\varphi_\delta(t,x) = M(\delta)v(t,M(\delta)^{-1} x) = M(\delta) F(t) \geq \sup_{h \in [0,\delta]} F(t-h) \geq F_\delta(t)\]
By uniform Lipschitz regularity of $v(t)$
\begin{align*}
u_\delta(t,x) &\leq \varphi_\delta(t,x) \\
&\leq M(\delta)v(t,x) +M(\delta)\|\grad v(t)\|_\infty(M(\delta)^{-1} -  1)\sup_{x \in \Omega(v(t))}|x|\\
&\leq v(t,x) + C(M(\delta)-1)\|F\|_\infty.
\end{align*}
Note that $\Omega(v(t)) \subset B_{CF(t)}(0)$ by uniform non-degeneracy of $v(t)$ and the upper bound $v(t) \leq F(t)$.  

This proves the upper bound. The lower bound is similar comparing with an inward dilation of size $m(\delta) = \inf_t\inf_{h \in [0,\delta]} F(t-h)/F(t)$ instead.

Furthermore, applying \cite[Th.~1.3
]{FKPi}, we also know that, along a subsequence, $u_\delta(t) \to u(t)$ pointwise in time and uniformly in space and $u$ is an energy solution.  By uniqueness of pointwise limits $u \equiv v$ and so the obstacle solution is an energy solution.
\end{proof}

\appendix

\section{Technical results}

\subsection{Lipschitz and non-degeneracy}
First we record the Lipschitz continuity of supersolutions of the one-phase problem. See \cite[Lemma 11.19]{CaffarelliSalsa} for a proof.

\begin{lemma}[Lipschitz estimate]\label{l.one-phase-lipschitz}
      There is $C(d) \geq 1$ so that for any viscosity solution $u$  of
\[\Delta u = 0 \ \hbox{ in } \ \Omega(u) \cap B_2 \ \hbox{ and } \ |\grad u|^2 \leq {Q} \ \hbox{ on } \ \partial \Omega(u) \cap B_2\]
 we have
\[\|\grad u\|_{L^\infty(B_1)} \leq C(\sqrt{Q}+\|\grad u\|_{L^2(B_2)}).\]
\end{lemma}

Non-degeneracy of Bernoulli problem solutions is much more delicate than Lipschitz continuity, it typically needs a stronger condition than just the viscosity solution property.  Here we give several cases where it is known.
\begin{lemma}\label{l.lip-bdry-nondegen}
Let $u \in C(B_2)$ satisfy one of the following: 
\begin{enumerate}[label = (\roman*)]
\item\label{part.lip-bdry-case} $\partial\{u>0\}$ is an $L$-Lipschitz graph in $B_2$ and $u$ solves
\[ \Delta u = 0 \ \hbox{ in } \ \{u>0\} \cap B_2, \ \hbox{ and } \  |\grad u| \geq 1 \hbox{ on } \ \partial \{u>0\} \cap B_2.\] 
\item\label{part.inward-min-case}  $u$ is an inward minimizer of $\mathcal{J}$ in $B_2$. 
\item\label{part.minimal-super-case} Suppose that $u$ is a minimal supersolution in $B_2$
\[ \Delta u = 0 \ \hbox{ in } \ \{u>0\} \cap B_2, \ \hbox{ and } \  |\grad u| \leq 1 \hbox{ on } \ \partial \{u>0\} \cap B_2.\] 
\item\label{part.maximal-sub-d2}  $d=2$ and $u$ is a maximal subsolution in $B_2$ of
\[ \Delta u = 0 \ \hbox{ in } \ \{u>0\} \cap B_2, \ \hbox{ and } \  |\grad u| \geq 1 \hbox{ on } \ \partial \{u>0\} \cap B_2.\] 
 \end{enumerate}
 Then there is $c_0$ depending on $d$ (and on $L$ if in the first case) such that  
 \[
 \sup_{B_r(x)} u(x) \geq c_0 r \quad \hbox{ for all } x\in \partial \{u>0\}\cap B_1, \ B_r(x) \subset B_2. 
\]
\end{lemma} 
\begin{proof}
For part \partref{lip-bdry-case} see \cite[Lemma 6.1]{DeSilva}, for part \partref{inward-min-case} see \cite[Lemma 4.4]{VelichkovBook}, for part \partref{minimal-super-case} see \cite[Lemma 6.9]{CaffarelliSalsa}, and for \partref{maximal-sub-d2} see \cite{Betul}.
\end{proof}

\subsection{Discussion of viscosity solution notions}
\label{s.app-vs-equivalence}

In \sref{comparison} we defined the spatial slope conditions in terms of the sub and superdifferential notions in \dref{subdifferentials-vs-def}. These appear a bit different from the classical sub and supersolution notions:

\begin{definition}\label{d.test-fcns-vs-def}
Suppose $u: U \to [0,\infty)$ is continuous and is subharmonic (resp. superharmonic) in $\Omega(u) \cap U$.  Let $G \subset \partial \Omega(u) \cap U$ be a relatively open subset. Then $u$ is a subsolution (resp. supersolution) of
\[|\grad u|^2 \geq {Q} \quad \hbox{(resp. $|\grad u|^2 \leq {Q}$)}\quad \hbox{ on } G\]
if, whenever $\varphi$ is a smooth test function such that $\varphi_+$ touches $u$ from above (resp. $\varphi$ touches $u$ from below) at $x \in G$ with $\Delta \varphi(x) < 0$ (resp. $\Delta \varphi(x) > 0$),
\[|\grad \varphi(x)|^2 \geq {Q} \qquad \hbox{(resp. $|\grad \varphi(x)|^2 \leq {Q}$)}.\]
\end{definition}

There is an equivalence between these notions.

\begin{lemma}\label{l.vs-two-notions-equiv}
Suppose that $u: U \to [0,\infty)$ is continuous and subharmonic (resp. superharmonic) in $\Omega(u) \cap U$ and let $G$ be a relatively open subset of $\partial \Omega(u) \cap U$. Then $u$ is a viscosity subsolution (resp. supersolution) in the sense of \dref{subdifferentials-vs-def} on $G$ if and only if it is a viscosity subsolution (resp. supersolution) in the sense of \dref{test-fcns-vs-def} on $G$.
\end{lemma}

\begin{remark}
Note that this equivalence requires both solution properties to hold \emph{everywhere}. If $\varphi$ touches $u$ from above at $x_0 \in \partial \Omega(u) \cap U$ then $\grad \varphi(x_0) \in D_+u(x_0)$. However the reverse is not true, if $p \in D_+u(x_0)$ it does not necessarily guarantee that $u$ can be touched from above by a smooth test function at the \emph{same point} with $\grad \varphi(x_0) \approx p$, rather one can find a nearby point where the touching occurs.

In \cite[Sec.~4
]{FKPi} this distinction is important since we only prove a sub- and superdifferential notion $\mathcal{H}^{d-1}$-almost everywhere.
\end{remark}

\begin{proof}[Proof of \lref{vs-two-notions-equiv}]
    We just do the subsolution equivalence, the supersolution one is similar. Suppose $u$ is a subsolution in the superdifferential sense \dref{subdifferentials-vs-def}. Suppose $\varphi$ is a smooth test function with $\varphi_+$ touching $u$ from above at $x_0 \in  G$.  Then 
\[u(x) \leq [\grad \varphi(x_0) \cdot (x - x_0)]_+ + o(|x-x_0|) \ \hbox{ so } \ \grad \varphi(x_0) \in D_+u(x_0)\]
and so
\[|\grad \varphi(x_0)|^2  \geq {Q}.\]

Next suppose that $u$ is a subsolution in the touching test function sense \dref{test-fcns-vs-def}, and $p \in D_+u(x_0)$ for some $x_0 \in G$.  Without loss assume that $p = \alpha e_d$. Let $\delta,r>0$ and define
\[ \varphi(x) : = (1+\delta) \alpha (x-x_0) \cdot e_d + c(d)\delta r^{-1}(|(x-x_0)'|^2  -  d x_d^2).\]
This function is a superharmonic polynomial and using the superdifferential condition $\alpha e_d \in D_+u(x_0)$ one can show that for all $\delta>0$ there is $r>0$ so that $u < \varphi$ on $\partial C \cap \overline{\{u>0\}}$ where $C$ is a small open cylindrical neighborhood of $x_0$ with radius $\approx r_0$.  Making the radius smaller if necessary we can assume that $\partial \Omega(u) \cap C \subset G$.

Then, since $\varphi(x_0) = u(x_0)$ we can shift $\varphi$ vertically $\varphi(x) + s$ until $(\varphi(x)+s)_+$ touches $u$ from above at some $x_1 \in C$. Since $\varphi$ is strictly superharmonic the touching must be on $x_1 \in \partial \Omega(u) \cap C \subset G$.  Thus finally the viscosity solution condition implies that
\[{Q} \leq |\grad \varphi(x_1)|^2 \leq \alpha + o_\delta(1).\]
Sending $\delta \to 0$ finishes the argument. 
\end{proof}

\subsection{Uniqueness of the Bernoulli problem solution in the complement of a star-shaped set}

In the star-shaped case we have the following uniqueness result for a Bernoulli problem with an obstacle.
The result without obstacle for classical solutions in two dimensions was proved in \cite{tepper1975} based on an existence theorem by Beurling \cite{Beurling} (more easily found reprinted here \cite{Beurling2}). Here we give a proof for viscosity solutions of the obstacle problem for convenience.

\begin{theorem}
\label{t.bernoulli-obstacle-uniqueness}
Suppose that $U$ is a domain such that $K = \R^d \setminus U$ is strongly star-shaped with respect to a neighborhood of the origin, and that $O$ is a strongly star-shaped set with respect to the same neighborhood. If $u, v \in C(\overline U)$ have compact support and for some constants $F, q > 0$ are viscosity solutions of a Bernoulli problem with obstacle $O$ from below
\begin{align*}
\left\{
\begin{aligned}
\Delta u &= 0 & & \text{in } \{u > 0\} \cap U,\\
u &> 0 && \text{in } O,\\
|\nabla u| &= q & & \text{on } (\partial \{u > 0\} \setminus \overline O) \cap U,\\
|\nabla u| &\leq q & & \text{on } \partial \{u > 0\} \cap U,\\
u &= F && \text{on } \partial U,
\end{aligned}
\right.
\end{align*}
then $u = v$.

The same result applies for a Bernoulli problem with obstacle $O$ from above.
\end{theorem}

\begin{proof}
By a simple rescaling it is sufficient to consider $F = q = 1$.
We show that $\Omega(u) = \Omega(v)$ which is sufficient by the uniqueness of the Laplace equation.

We first check that $\Omega(u) \subset \Omega(v)$. Suppose that $\Omega(u) \setminus \Omega(v) \neq \emptyset$. Consider the largest $\lambda < 1$ such that 
\begin{align*}
v^\lambda(x) := v(\lambda x)
\end{align*}
satisfies
\begin{align*}
\Omega(u) \subset \Omega(v^\lambda) = \lambda^{-1} \Omega(v).
\end{align*}
Such $\lambda$ exists since both $\Omega(u)$ and $\Omega(v)$ are open, $\Omega(u)$ is bounded and $\Omega(v)$ contains a neighborhood of the origin. We have clearly $\lambda < 1$, and by the definition of $\lambda$ there exists $x_0 \in \partial \Omega(u) \cap \partial \Omega(v^\lambda)$.
Also note that $x_0 \notin \overline O$ since $O$ is strongly star-shaped and $O \cap \partial \Omega(v) = \emptyset$.

If $u$ and $v$ are sufficiently smooth, we have
\begin{align*}
|\nabla v^\lambda|(x_0) = \lambda|\nabla v|(\lambda x_0) \leq \lambda < 1 = |\nabla u|(x_0),
\end{align*}
a contradiction with the fact that $u \leq v^\lambda$ by the maximum principle for the harmonic function $v^\lambda - u$ in $\Omega(v) \cap \lambda^{-1}U$ since $v^\lambda - u > 0$ on $\lambda^{-1} \partial U \subset U$ by the strong star-shapedness.

To make the argument work without assuming regularity, we consider the sup and inf convolutions
\begin{align*}
v_r(x) = \inf_{\overline B_r(x)} v, \qquad u_r(x) = \sup_{\overline B_r(x)} u
\end{align*}
defined on $\overline U_r$ where $U_r:= \{x: \overline B_r(x) \subset U\}$.
It is easy to see that $v_r$ is superharmonic on $U_r \cap \Omega(v_r)$, $u_r$ is subharmonic on $U_r \cap \Omega(u_r)$, and $\inf_{\partial U_r} v_r \to 1$ as $r \to 0$ by continuity.

Let us call the $\lambda < 1$ above $\lambda_0$. For each $r > 0$ there exists a largest $\lambda_r < 1$ so that $\Omega(u_r) \subset \lambda_r^{-1}\Omega(v_r)$.
Furthermore $\lambda_r \to \lambda_0 < 1$ as $r \to 0$. Therefore we can choose $r > 0$ sufficiently small so that $\lambda_r < 1$, $\lambda_r^{-1/2} \inf_{\partial U_r} v_r > 1$ and $\dist(O, \lambda_r^{-1} \partial\Omega(v_r)) > r$ (by strong star-shapedness of $O$).

We fix such $r > 0$ and choose $x_0 \in \partial \Omega(u_r) \cap \lambda_r^{-1}\partial\Omega(v_r)$. Let us set $w(x) := \lambda_r^{-1/2}v_r(\lambda_r x)$. We have that both $u_r$ and $w$ are defined on $\lambda_r^{-1}\overline U_r$. We also note that $w - u_r$ is superharmonic on $\lambda_r^{-1} U_r \cap \{u_r > 0\}$ with $w - u_r > 0$ on $\lambda_r^{-1}\partial U_r$. Therefore the slopes of $u_r$ and $w$ in the normal direction at $x_0$ are ordered.

We have the standard situation of boundaries of $u_r$ and $w$ touching at a point that has both interior and exterior touching balls centered at the free boundaries of $u$ and $v(\lambda_r x)$ respectively. By a standard barrier construction at these centers, we arrive at a contradiction with the fact that $u$ satisfies $|\nabla u| \geq 1$ (since the center is not in $\overline O$ by the choice of $r$) and $v^{\lambda_r} := v(\lambda_r x)$ satisfies $\lambda_r^{-1/2}|\nabla v^{\lambda_r}| \leq \lambda_r^{1/2} 1 < 1$.
Therefore we conclude that $\Omega(u) \subset \Omega(v)$.

The inclusion $\Omega(v) \subset \Omega(u)$ can be shown by the same argument, swapping the roles of $u$ and $v$. By the uniqueness of the Laplace equation we have $u = v$.

The proof for the problem with obstacle from above follows an analogous argument.
\end{proof}

\bibliography{monotone-articles}

\end{document}